\documentclass[10pt]{amsart}
 \usepackage[foot]{amsaddr}
\usepackage{curly}
\usepackage{stmaryrd}
\usepackage{paperstyles}
\usepackage{tikz-cd}

\allowdisplaybreaks
\usepackage[backend=biber,sorting=nty,
style=alphabetic,maxnames=5]{biblatex}

\bibliography{references}
\AtNextBibliography{\small}
\renewbibmacro{in:}{}

\title{Airy structures and deformations of curves in surfaces}

\author{W. Chaimanowong} \email{cnw@student.unimelb.edu.au}
\author{P. Norbury} 
\email{norbury@unimelb.edu.au}
\author{M. Swaddle} 
\email{meswaddle@protonmail.com}
\author{M. Tavakol} \email{mehdi.tavakol@unimelb.edu.au}
\address{School of Mathematics and Statistics, University of Melbourne, VIC 3010, Australia}

\date{\today}

\begin{document}
    
    \begin{abstract}
        An embedded curve in a symplectic surface $\Sigma\subset X$ defines a smooth deformation space $\cb$ of nearby embedded curves.  A key idea of Kontsevich and Soibelman \cite{KSoAir} is to equip the symplectic surface $X$ with a foliation in order to study the deformation space $\cb$.   The foliation, together with a vector space $V_\Sigma$ of meromorphic differentials on $\Sigma$, endows an embedded curve $\Sigma$ with the structure of the initial data of topological recursion, which defines a collection of symmetric tensors on $V_\Sigma$.   Kontsevich and Soibelman define an Airy structure on $V_\Sigma$ to be a formal quadratic Lagrangian $\cl\subset T^*(V_\Sigma^*)$ which leads to an alternative construction of the tensors of topological recursion.  In this paper we produce a formal series $\theta$ on $\cb$ 
which takes it values in $\cl$, and use this to produce the Donagi-Markman cubic  from a natural cubic tensor on $V_\Sigma$, giving a generalisation of a result of Baraglia and Huang \cite{BHuSpe}.

    \end{abstract}
    
    \maketitle
    
    \renewcommand{\baselinestretch}{0.75}\normalsize
    \tableofcontents
    \renewcommand{\baselinestretch}{1.0}\normalsize    
    \section{Introduction}

Consider a smooth algebraic curve $\Sigma$ embedded in a symplectic algebraic surface $X$.  The purpose of this paper is to study the relation of the local deformation space of $\Sigma$ to topological recursion following Kontsevich and Soibelman \cite{KSoAir}.  

The local deformation space $\cb$ of $\Sigma$ parametrises embeddings of smooth curves near to $\Sigma\subset X$.  It is a smooth complex analytic space of dimension equal to the genus of $\Sigma$.  Represent $\Sigma\subset X$ by the point $[\Sigma]\in\cb$.  Over $\cb$ is a flat symplectic bundle $\ch\to\cb$ with fibres $H^1(\Sigma;\bc)$ and equipped with the Gauss-Manin connection $\nabla^{\text{GM}}$.  The natural linear embedding $H^0(\Sigma,K_\Sigma)\subset H^1(\Sigma;\bc)$, that sends a holomorphic differential to its cohomology class,  defines a Lagrangian subbundle of $\ch$ with fibres $H^0(\Sigma,K_\Sigma)$.  The symplectic structure on $X$ defines the exact sequence  
\begin{equation}  \label{ext}
    0 \rightarrow T \Sigma \rightarrow TX|_\Sigma \rightarrow K_\Sigma \rightarrow 0.
\end{equation}
In particular, the normal bundle $\nu_\Sigma$ to $\Sigma\subset X$ is isomorphic to $K_\Sigma$,
hence the tangent space $H^0(\Sigma,\nu_\Sigma)$ to $\cb$ at $[\Sigma]$ is isomorphic to the vector space of holomorphic differentials $H^0(\Sigma,K_\Sigma)$. The isomorphism is denoted 
\[\phi:T_{[\Sigma]}\cb\stackrel{\cong}{\longrightarrow} H^0(\Sigma,K_\Sigma)\]
which defines an $\ch$-valued 1-form $\phi\in\Gamma(\cb,\Omega^1_\cb\otimes\ch)$.  One can realise $\phi$ via variations of a section 
\[
[\theta]\in\Gamma(U_{[\Sigma]},\ch)
\]
where $U_{[\Sigma]}\subset\cb$ is a neighbourhood of $[\Sigma]$, and $[\theta]$ is
characterised by
\begin{equation}  \label{theta}
\phi=\nabla^{\text{GM}}[\theta],\qquad [\theta]([\Sigma])=0.
\end{equation}
A construction of $[\theta]$ is given in \eqref{cohtheta} in Section~\ref{sec:TR}.
Hence for $[\Sigma']\in U_{[\Sigma]}$ and any $v\in T_{[\Sigma']}\cb$, 
$\phi(v)=\nabla_v^{\text{GM}}[\theta]\in H^0(\Sigma',K_{\Sigma'})$
which defines $[\theta]$ uniquely up to addition of a constant section, and the ambiguity is removed by setting $[\theta]([\Sigma])=0$.   One can equivalently define $[\theta]$ via parallel transport of the flat connection on $\ch$ given by $\nabla^{\text{GM}}+\phi$.  The property $\nabla_v^{\text{GM}}[\theta]\in H^0(\Sigma',K_{\Sigma'})$ is a cohomological version of the property of a Seiberg-Witten differential. 

Equip $X$ with a Lagrangian foliation $\cf$, or equivalently a holomorphic sub-line-bundle $L_\cf\subset TX$.  More generally, the foliation may be singular at finitely many points so $L_\cf\subset TX$ is a subsheaf and $\Sigma$ is chosen to avoid these singular points.  For example, if the foliation is defined by the fibres of a morphism $\pi:X\to C$ to a curve $C$, then $L_\cf=\ker D\pi$ is not locally free at the critical points of $\pi$ contained in the singular fibres.  Define $R\subset\Sigma$ to be those points where $\Sigma$ meets $\cf$ tangentially, so by \eqref{ext} $L_\cf|_\Sigma\cong K_\Sigma(-R)$. 
Furthermore, we choose $\Sigma$ so that $R\subset \Sigma$ is finite and each tangent point is simple. 
The simple tangency condition is an open condition hence also true of any nearby curve $\Sigma'$ in the family $\cb$ and defines $R'\subset \Sigma'$.  A key idea of Kontsevich and Soibelman in \cite{KSoAir} is to use the Lagrangian foliation $\cf$ to study the deformation space $\cb$ via lifting cohomology classes in $H^1(\Sigma;\bc)$ to meromorphic differentials on $\Sigma$ with poles at $R\subset\Sigma$.  Define $G_\Sigma$ to be the vector space of residueless meromorphic differentials on $\Sigma$, holomorphic on $\Sigma-R$, and $\bg\to\cb$ the bundle with fibres $G_\Sigma$.  The map $G_\Sigma\to H^1(\Sigma,\bc)$ which sends a differential to its cohomology class is surjective and induces the surjective map of vector bundles $\bg\to\ch$.  

 Topological recursion, as defined by Eynard and Orantin \cite{eynard_orantin}, is a recursive procedure that produces from a spectral curve $S=(\Sigma,u,v,B)$ a symmetric tensor product of meromorphic 1-forms $\omega_{h,n}$ on $\Sigma^n$ for integers $h\geq 0$ and $n\geq 1$, which we refer to as  correlators.   Here, a spectral curve, $S=(\Sigma,u,v,B)$ is a curve $\Sigma$ equipped with two meromorphic functions $u, v: \Sigma\to \mathbb{C}$ holomorphic in a neighbourhood of points where $du=0$ and a bidifferential $B(p_1,p_2)$ defined in \eqref{bidiff}, such that $du$ has only simple zeros. More generally, $u$ and $v$ need only to be locally defined.

A curve $\Sigma\subset (X,\cf)$ together with a choice of $a$ and $b$-cycles that form a Torelli basis $\{a_1,...,a_g,b_1,...,b_g\}\subset H_1(\Sigma;\bz)$ produces a spectral curve $S=(\Sigma,u,v,B)$ and hence the initial data of topological recursion.  The $a$ and $b$-cycles on $\Sigma$ uniquely determine a bidifferential $B(p_1,p_2)$ defined on $\Sigma$---see \eqref{bidiff}. The locally defined functions are restrictions to $\Sigma$ of locally defined coordinates $u$ and $v$ on $X$ chosen so that the symplectic form on $X$ is $\omega=du\wedge dv$ and the leaves of the foliation are defined via $u=$ constant---denoted {\em foliation-Darboux} coordinates in Definition~\ref{FDcoord}.    

The choice of $a$-cycles 
defines 
\[ V_\Sigma=\{\eta\in G_\Sigma\mid\oint_{a_i}\eta=0,i=1,...,g\}\]
which has trivial intersection with $H^0(\Sigma,K_\Sigma)$ since non-trivial holomorphic differentials cannot have zero $a$-periods.  
The choice of Torelli basis extends to a well-defined choice on a neighbourhood $U_{[\Sigma]}\subset\cb$ of $[\Sigma]$, hence $V_\Sigma$ defines a subbundle of $\bg$ on $U_{[\Sigma]}$.
The correlator $\omega_{h,n}(p_1,p_2,...,p_n)$ is symmetric in $p_i$ and has poles only at $p_i\in R\subset\Sigma$ with zero residues and vanishing $a$-periods.  In other words it lives in the $n$-th symmetric power
\[\omega_{h,n}(p_1,p_2,...,p_n)\in \mathrm{S}^n(V_\Sigma)\]
where \(\mathrm{S}^n\) is the \(n\)-th symmetric algebra.

For any residueless meromorphic differential $\eta$ defined on $\Sigma$, denote its normalised periods by:
\begin{equation}   \label{bnorm}
\oint_{\hat{b}_k}\eta:=-\frac{1}{2\pi i}\oint_{b_k}\eta.
\end{equation}
For $[\Sigma]\in\cb$, the functions \[z^k=\oint_{a_k}[\theta],\quad k=1,...,g\] define coordinates on a neighbourhood $U_{[\Sigma]}\subset\cb$ and also on a formal neighbourhood $\widehat{\cb}_{[\Sigma]}$ of $[\Sigma]\in\cb$. 

Given a curve $\Sigma\subset (X,\cf)$ together with a choice of $a$ and $b$-cycles, which determine normalised holomorphic differentials $\omega_i\in H^0(\Sigma,K_\Sigma)$ and correlators $\omega_{h,n}$, the following theorem constructs a formal series of meromorphic differentials that lives above the local analytic expansion of the section $[\theta]$. 
\begin{theorem}     \label{thetatheorem}
Define a section $\theta\in\Gamma(\widehat{\cb}_{[\Sigma]},G_\Sigma)$ by
\begin{equation}  \label{thetaTR1}   
\theta=z^i\omega_i-\frac12z^iz^j\oint_{\hat{b}_i}\oint_{\hat{b}_j}\omega_{0,3}-\frac{1}{3!}z^iz^jz^k\oint_{\hat{b}_i}\oint_{\hat{b}_j}\oint_{\hat{b}_k}\omega_{0,4}-...
\end{equation}
where we sum over indices in $\{1,...,g\}$.  Its cohomology class in $\Gamma(\widehat{\cb}_{[\Sigma]},\ch)$ is analytic in $z^1,...,z^g$ and coincides with the analytic expansion of $[\theta]$ defined in \eqref{theta}.
\end{theorem}  
We use the convention of summation over repeated indices and throughout the paper, except when we wish to emphasise the indices.
Note that the cohomology class $[\theta]$ can be naturally expressed by its periods
\[[\theta]=(\oint_{a_i}[\theta],\oint_{b_i}[\theta]\mid i=1,...,g)\in\bc^{2g} \llbracket z^1,...,z^g \rrbracket .\]
Properties of the series \eqref{thetaTR1} leads to relations among residues and periods of $\omega_{0,h}$, such as \eqref{relat} and \eqref{relat2}.

The series $\theta$ is a formal expansion  of the Seiberg-Witten differential in the Seiberg-Witten family of curves \cite{NOkSei}.  Similarly, in the case $X=T^*C$ the formal series $\theta$ is a formal expansion of the tautological 1-form on $T^*C$---see \eqref{eq:thetader} in Section~\ref{formalconv} for a precise statement. An analytic expansion of the Seiberg-Witten differential or tautological 1-form would require a natural local trivialisation of the bundle $\bg$.  The foliation produces a flat connection on $\bg$, defined in Section~\ref{formalconv}, however this does not produces a local trivialisation since parallel transport  for this connection is not defined.

The period matrix $\tau_{ij}$ of $\Sigma$ appears in the first order terms of $[\theta]$ via \[\oint_{b_i}[\theta]=z^j\tau_{ij}|_\Sigma+ \text{ higher order terms}\]
which leads to the following corollary.
\begin{corollary}  \label{om3cor}
The variation of the period matrix of a curve $\Sigma\subset X$ is
\[ \frac{\partial\tau_{ij}}{\partial z^k}= -2\pi i \oint_{\hat{b}_i}\oint_{\hat{b}_j}\oint_{\hat{b}_k}\omega_{0,3}\]
and more generally
\[\frac{\partial^n\tau_{ij}}{\partial z^{i_1}..\partial z^{i_n}}=-2\pi i \oint_{\hat{b}_i}\oint_{\hat{b}_j}\oint_{\hat{b}_{i_1}}...\oint_{\hat{b}_{i_n}}\omega_{0,n+2}.\] 
\end{corollary}
When $X=T^*C$, Corollary~\ref{om3cor} was proven by Baraglia and Huang in \cite{BHuSpe}, and Bertola and Korotkin \cite{BKoSpa}. 

Corollary~\ref{om3cor} can be packaged into an expression for an analytic expansion of the prepotential $F_0:\cb\to\bc$ defined in Section~\ref{sec:def}, such as the Seiberg-Witten prepotential \cite{NOkSei}, in terms of periods of the correlators:
\begin{equation}  \label{prepot}
    F_0(z^1,...,z^g)= -\sum_{I}\frac{z^I}{|I|!}\oint_{\hat{b}_I}\omega_{0,|I|}.
\end{equation}
The summation is over multi-indices $I=(i_1,...,i_n)$, for $i_k\in\{1,...,g\}$ and the integral is $\oint_{\hat{b}_I}\omega_{0,|I|}=\oint_{\hat{b}_{i_1}}...\oint_{\hat{b}_{i_n}}\omega_{0,n}(p_1,...,p_n)$.

Note that a choice of $a$-cycles on $\Sigma\subset(X,\Omega)$ determines local coordinates $\{z^1,...,z^g\}$ on the deformation space $\cb$ of $\Sigma$ and the prepotential $F_0:\cb\to\bc$, well-defined up to quadratic terms in $z^i$.  Moreover, $F_0$ depends (up to quadratic terms) only on the linear sub-module $L_a\subset H_1(\Sigma;\bz)$ spanned by the $a$-cycles. This is reflected clearly in \eqref{prepot} since the difference between two choices of $b$-cycles is an element of $L_a$.  The correlators $\omega_{0,|I|}$ in \eqref{prepot} vanish on $L_a$ for $|I|\geq 3$ so only the quadratic term involving $\omega_{0,2}$ detects a change in $b$-cycles.

The series $\theta$ defined in \eqref{thetaTR1} has a geometric interpretation which we now describe.  Kontsevich and Soibelman \cite{KSoAir} formulated topological recursion in terms of an {\em Airy structure} which characterises a quadratic Lagrangian $\cl\subset W$ in a symplectic vector space, i.e. a sub-variety defined by polynomials of degree $\leq 2$.  A basic example is the plane conic tangent to the line $y=0$:
\[\cl=\{-y+ax^2+2bxy+cy^2=0\}\subset\bc^2.\]
More generally, consider a finite-dimensional symplectic vector space $W\cong\bc^{2N}$ and a quadratic Lagrangian $\cl\subset W$ containing $0$, with tangent space $L=T_0\cl$.  Choose a Lagrangian complement $V$ to $L$ in $W$
    \[W=L\oplus V.\]
Note that the exact sequence $L\to W\stackrel{\Omega_W(\cdot,\cdot)}{\to} L^*$ produces a canonical isomorphism $V\cong L^*$ hence a canonical isomorphism $W\cong V^*\oplus V=T^*(V^*)$ so that $V$ is a polarisation of $W$. 

Choose Darboux coordinates $\{x^i,y_i\}_{i=1,...,N}\in W^*$, i.e. $\Omega_W=dx^i\wedge dy_i$, with $x^i\in V$ and $y_i\in L$ which are naturally coordinates on the base $V^*$, respectively fibre $V$, of $T^*(V^*)$.
We have $L=\{y_i=0\}$ and the Lagrangian $\cl$ is defined by
    \[\cl=\{H_i=0\mid i=1,...,N\}\]
    for
    \[ H_i=-y_i+a_{ijk}x^jx^k+b_{ij}^kx^jy_k+c_i^{jk}y_jy_k,\quad i=1,...,N.
    \]

    The coefficients of the $H_i$ are tensors on $V$.
    \begin{alignat*}{3}
    A &=(a_{ijk})  &&\in V\otimes V\otimes V,\nonumber\\
    B &=(b_{ij}^k) &&\in V^*\otimes V\otimes V,\\   
    C &=(c_i^{jk}) &&\in V^*\otimes V^*\otimes V,\nonumber
    \end{alignat*}
    where $\displaystyle(a_{ijk}):=a_{ijk}x^ix^jx^k$, $\displaystyle(b_{ij}^k):=b_{ij}^kx^ix^jy_k$ and $\displaystyle(c_i^{jk}):=c_i^{jk}x^iy_jy_k$ and as usual we sum over the indices $i,j,k$.

    The defining functions of the Lagrangian submanifold satisfy
    \begin{equation}  \label{hamlag}
    \{H_i,H_j\}=g_{ij}^kH_k
    \end{equation} 
    where $g_{ij}^k$ are functions in general, but numbers here since $H_i$ are quadratic. The relation \eqref{hamlag} implies a collection of conditions on the tensors $A$, $B$ and $C$. The linear term contribution to \eqref{hamlag} implies $A\in \mathrm{S}^3(V)$, whereas {\em a priori} $A$ is symmetric only in its final two arguments, and it also implies $g_{ij}^k=2b_{ji}^k-2b_{ij}^k$. The remaining conditions, corresponding respectively to vanishing of coefficients $x^kx^m$, $x^ky^m$ and $y^ky^m$ in \eqref{hamlag}, are homogeneous of degree two in the tensors and given explicitly in Definition~\ref{defairy}.
    
Kontsevich and Soibelman define an Airy structure to be a collection of tensors on a  vector space $V$:
\begin{align*}
    A\in \mathrm{Sym}^3(V),\quad B\in V^{*}\otimes V\otimes V,\quad C\in \mathrm{Sym}^2(V^{*})\otimes V,
\end{align*}
satisfying the quadratic relationships implied by \eqref{hamlag}.  An Airy structure makes sense for infinite dimensional $V$.  In finite dimensions an Airy structure is equivalent to a quadratic Lagrangian submanifold of the symplectic vector space $T^*(V^*)$, while in infinite dimensions it corresponds to a formal Lagrangian submanifold defined by the ideal generated by a collection of quadratic polynomials $H_i$, $i=1,2,....$---see Section~\ref{airystr}.

A fundamental example of a formal Lagrangian subvariety in an infinite dimensional symplectic vector space 
 arises from Virasoro relations satisfied by the Kontsevich-Witten tau function of the KdV hierarchy
$$ Z^{\text{KW}}(\hbar,x^1,x^3,...)=\exp \left( \sum_{h,n,\vec{k}}\frac{\hbar^{h-1}}{n!}\int_{\overline{\cm}_{h,n}} \prod_{i=1}^n\psi_i^{k_i}(2k_i+1)!!x^{2k_i+1}\right)
$$
which is a generating function for intersection numbers of tautological classes $\psi_i$ on the moduli space of stable curves.   Define 
$\{L_{-1}, L_0, L_1, \ldots\}$ which satisfy the Virasoro commutation relations
\[
[L_m, L_n] = (m-n) L_{m+n}, \quad \text{for } m, n \geq -1
\]
by
$$L_m=-\frac{1}{2}\frac{\partial}{\partial x^{2m+3}} +\frac{\hbar}{4} \hspace{-1mm}\mathop{\sum_{i+j=2m}}_{i,j\text{ odd}} \hspace{-2mm} \frac{\partial^2}{\partial x^i \partial x^j}
 +\frac{1}{2}\mathop{\sum_{i=1}}_{i\text{ odd}}^\infty i x^i \frac{\partial}{\partial x^{i+2m}}+ \frac{1}{16} \delta_{m,0}+\frac{(x^1)^2}{4\hbar}\delta_{m,-1}
$$
where the sum over $i+j=2m$ is empty when $m=0$ or $-1$ and $\frac{\partial}{\partial x^{-1}}$ is the zero operator. Then 
\[L_mZ^{\text{KW}}(\hbar,x^1,3x^3,...)=0,\quad m\geq -1\]
which uniquely determines any intersection number recursively from the initial calculation $\int_{\overline{\cm}_{0,3}} 1=1$ as conjectured by Witten  \cite{WitTwo} and proven by Kontsevich  \cite{KonInt}.  
Define the symplectic vector space of residueless Laurent series
\begin{equation}  \label{wairy}
    W_{\text{Airy}}=\left\{ J=\sum_{n\in\bz} J_nz^{-n}\frac{dz}{z}\mid J_0=0, \exists N \text{ such that }J_n=0,\ n>N\right\}
\end{equation} 
with symplectic form
$$\Omega_W(\eta_1,\eta_2)=\Res_{z=0}f_1\eta_2,\quad df_1=\eta_1,\eta_2\in W.$$
There is a symplectomorphism 
$W_{\text{Airy}}\cong\mathrm{Spf}(\mathbb{C}\llbracket x^{\sbt},y_{\sbt} \rrbracket )$ 
equipped with the Poisson bracket $\{x^i,y_j\}=\delta^i_{j}$ and $\{x^i,x^j\}=0=\{y_i,y_j\}$, $i,j=1,...,\infty$.

\begin{example}   \label{quadlagKW}
Define the quadratic Lagrangian 
$$\cl_{\text{Airy}}\subset W_{\text{Airy}}=\mathrm{Spf}(\mathbb{C}\llbracket x^{\sbt},y_{\sbt} \rrbracket )$$ 
via the ideal generated by the linear and quadratic functions 
\begin{align*}
   H_k(x^{\sbt},y_{\sbt})&=-y_k,\quad k\in\bz^+_{\text{even}}\\
H_k(x^{\sbt},y_{\sbt})& = \hbar L_{\frac{k-3}{2}}\left(x^{\sbt},\hbar\frac{\partial}{\partial x^{\sbt}}\right)|_{\hbar\frac{\partial}{\partial x^i}=y_i}\quad k\in\bz^+_{\text{odd}}\\
&=-\tfrac12 y_k +\tfrac14 \hspace{-2mm}\mathop{\sum_{i+j=k-3}}_{i,j\text{ odd}} \hspace{-2mm} y_i y_j
 +\tfrac12\mathop{\sum_{i=1}}_{i\text{ odd}}^\infty i x^i y_{i+k-3}+ \tfrac{1}{16} \delta_{k,3}+\tfrac{1}{4}\delta_{k,1}(x^1)^2
\end{align*}
where $y_{-1}=0$.
\end{example}

The local behaviour of the topological recursion correlators at each point of $R\subset\Sigma$, gives rise to the tau function $Z^{\text{KW}}(\hbar,t_0,t_1,...)$ corresponding to the quadratic Lagrangian $\cl_{\text{Airy}}$.

There are natural embeddings $V_\Sigma\subset G_\Sigma\subset W_{\text{Airy}}^R$ defined by identifying $W_{\text{Airy}}$ with local residue-free differentials---see \eqref{gsigma} and \eqref{vsigma}---and sending global meromorphic differentials to their local expansions at each point in $R$ with respect to a given local coordinate.  We have $T^*(V_\Sigma^*)\cong W_{\text{Airy}}^R$ as symplectic vector spaces.  The (formal) quadratic Lagrangian submanifold produces an Airy structure on $V_\Sigma$:
\begin{equation}  \label{quadlagINTRO}
    \cl_{\text{KS}}=\cl_{\text{Airy}}^R\subset T^*(V_\Sigma^*)\leadsto(A_\Sigma,B_\Sigma,C_\Sigma).
\end{equation}
The dependence of the Airy structure on $\Sigma\subset X$ is through the polarisation of $W_{\text{Airy}}^R\supset\cl_{\text{KS}}$. The embedding $G_\Sigma\subset W_{\text{Airy}}^R$ is coisotropic and the quotient becomes a symplectic quotient $$H^1(\Sigma;\bc)=G_\Sigma/G_\Sigma^\perp=:W_{\text{Airy}}^R\hspace{-.5mm}\sslash\hspace{-1mm} G_\Sigma^\perp.$$
The image under the quotient map of the tensor $A_\Sigma\in V_\Sigma\otimes V_\Sigma\otimes V_\Sigma\to\overline{V}_\Sigma\otimes\overline{V}_\Sigma\otimes\overline{V}_\Sigma$ is rather natural.

The section $\theta\in\Gamma(\widehat{\cb}_{[\Sigma]},G_\Sigma)$ constructed in Theorem~\ref{thetatheorem} takes its values in $\cl_{\text{KS}}$---see Proposition~\ref{phisec1}.  This is used to understand the relation of the Airy structure built out of $\Sigma\subset(X,\Omega_X,\cf)$ to the local geometry of the space $\cb$, stated concretely in Theorem~\ref{main} below.

The {\em Donagi-Markman cubic} \cite{DMaCub} is the extension class defined by the exact sequence \eqref{ext}, which gives rise to a tensor on $\cb$:
$$\bar{A}_\Sigma\in\mathrm{Ext}^1(K_\Sigma,T \Sigma) \cong H^0(\Sigma,K_\Sigma^{\otimes 3})^\vee\cong \left(T^*_{[\Sigma]}\cb\right)^{\otimes 3}.$$

There is a natural isomorphism $H^0(\Sigma,K_\Sigma)^\vee\cong \overline{V}_\Sigma$ where $\overline{V}_\Sigma$ is the image of $V_\Sigma$ under the quotient map $G_\Sigma\to\ch_\Sigma$:
\[\overline{V}_\Sigma=\{\eta\in H^1(\Sigma;\bc)\mid\oint_{a_i}\eta=0,i=1,...,g\}\]  
which satisfies
$$H^1(\Sigma;\bc)=H^0(\Sigma,K_\Sigma)\oplus\overline{V}_\Sigma.$$
Any complement to $H^0(\Sigma,K_\Sigma)$ is naturally isomorphic to $H^0(\Sigma,K_\Sigma)^\vee$ via the symplectic form on $H^1(\Sigma;\bc)$.
Via this natural isomorphism, the Donagi-Markman cubic is represented by
$$\bar{A}_\Sigma\in\overline{V}_\Sigma\otimes\overline{V}_\Sigma\otimes\overline{V}_\Sigma.
$$
\begin{theorem}  \label{main}
The image of the tensor $A_\Sigma$ under the quotient map $V_\Sigma\to\overline{V}_\Sigma$  is the tensor $\bar{A}_\Sigma$.
\begin{align}  \label{tensquot}
    V_\Sigma\otimes V_\Sigma\otimes V_\Sigma&\to\overline{V}_\Sigma\otimes\overline{V}_\Sigma\otimes\overline{V}_\Sigma\\
    A_\Sigma&\mapsto\bar{A}_\Sigma.\nonumber
\end{align}
\end{theorem}
The Donagi-Markman cubic can be calculated via variations $\frac{\partial\tau_{ij}}{\partial z^k}$ and hence \eqref{tensquot} can be deduced from Corollary~\ref{om3cor} together with the result $A_\Sigma=\omega_{0,3}$---see Proposition~\ref{A=w3}.  Instead, we give a direct, geometric proof of Theorem~\ref{main}.  In Section~\ref{geomA}, $A_\Sigma$ is constructed as a linear map \[T_0\cl_{\text{KS}}\otimes T_0\cl_{\text{KS}}\to V_\Sigma\]
via covariant differentiation of a vector field on $\cl_{\text{KS}}$.  
The tensor $\bar{A}_\Sigma$ similarly arises via covariant differentiation of vector fields. Any vector $v\in T_{[\Sigma]}\cb\cong H^0(\Sigma,K_\Sigma)$ extends locally to a unique vector field $\tilde{v}\in\Gamma(U_{[\Sigma]},T\cb)]\subset\Gamma(U_{[\Sigma]},\ch)$ defined by requiring $\oint_{a_i}\tilde{v}$ to be locally constant.  
The covariant derivative $\nabla^{\text{GM}}_u\tilde{v}$ lives inside $\overline{V}_\Sigma$, since the derivative of constant $a$-periods is zero. Hence $\nabla^{\text{GM}}_u\tilde{v}$ takes in two vectors $u,v\in T_{[\Sigma]}\cb$ and defines a linear map \[H^0(\Sigma,K_\Sigma)\otimes H^0(\Sigma,K_\Sigma)\to\overline{V}_\Sigma\] 
which is identified with $\bar{A}_\Sigma$.
Thus, both $A_\Sigma$ and $\bar{A}_\Sigma$ are obtained via covariant differentiation with respect to a flat connection of a tangent vector field by a tangent vector.  Moreover, the vector fields and flat connection upstairs are  related to vector fields and flat connection downstairs.  To implement this idea one needs to use the formal germ of a Lagrangian and formal vector fields upstairs, together with the linearisation of $\theta$ defined in Theorem~\ref{thetatheorem}.  One consequence of Theorem~\ref{main} is that although $A_\Sigma$ is constructed only in a formal neighbourhood of a point in $\cb$, it descends to an analytic tensor which extends over all of $\cb$.

It is interesting that the methods used here, following Kontsevich and Soibelman, embed $\cb$ into a vector space of meromorphic differentials on the curve, with poles located at the branch points on the spectral curve, while the methods used by Bertola and Korotkin to prove Corollary~\ref{om3cor} embed $\cb$ into a moduli space of meromorphic differentials with poles located at the poles of the Higgs field.  They write in \cite{BKoSpa}: ``This suggests a possibility of the existence of a natural simple structure on spaces of abelian differentials which underlie the topological recursion framework on spaces of spectral covers".   Indeed the methods of Kontsevich and Soibelman produce topological recursion from a natural structure on the space of meromorphic differentials on a curve.   Although the meromorphic differentials differ in both cases, it would be interesting to compare these two approaches.

In Section~\ref{sec:TR} we define topological recursion for any smooth curve embedded in a foliated symplectic surface $\Sigma\subset(X,\Omega_X,\cf)$, and give examples of foliated symplectic surfaces.  Topological recursion is related to cohomological field theories \cite{DOSSIde} and we describe its consequences for the deformation space $\cb$ of $\Sigma\subset X$ in Section~\ref{cohft}.  In Section~\ref{quadlag} we define the approach to topological recursion by Kontsevich and Soibelman \cite{KSoAir}.  The quadratic Lagrangian used is $\cl_{\text{Airy}}$ constructed from the Kontsevich-Witten tau function.  In Section~\ref{BGW} we instead use a quadratic Lagrangian $\cl_{\text{Bessel}}$ built from the Br\'{e}zin-Gross-Witten tau function of the KdV hierarchy.  In Section~\ref{formalconv} we define the series $\theta\in\Gamma(\widehat{\cb}_{[\Sigma]},G_\Sigma)$ defined in any formal neighbourhood of $[\Sigma]\in\cb$ and prove its properties.  Appendix~\ref{variation} contains a proof of the variation formula for the correlators $\omega_{h,n}$ due to Eynard and Orantin \cite{eynard_orantin} and adapted to the spectral curves arising out of $\Sigma\subset(X,\Omega_X,\cf)$.

{\em Acknowledgements.}   The authors would like to thank David Baraglia, Todor Milanov, Jan Soibelman and Kari Vilonen for useful conversations.  WC and PN would like to thank the Max Planck Insitute for Mathematics, Bonn, where part of this work was carried out and PN would like to thank LMU, Munich where part of this work was carried out.  This work was partially supported under the  Australian
Research Council {\sl Discovery Projects} funding scheme project number DP180103891. 

    \section{Topological recursion applied to curves in surfaces}  \label{sec:TR}
In this section we apply topological recursion as defined in \cite{eynard_orantin} to $$\Sigma \subset(X,\Omega_X,\cf)$$ 
given by a compact curve embedded inside a (holomorphic) symplectic surface $(X,\Omega_X)$ with Lagrangian foliation $\cf$ following Kontsevich and Soibelman \cite{KSoAir}.  We begin with a description of the prepotential on the deformation space of $\Sigma$ inside $X$.  We then equip the surface $X$ with a Lagrangian foliation, $\cf$, which puts the extra structure on $\Sigma$ required to define a spectral curve which is the initial data of topological recursion.   More generally one should be able to relax the symplectic condition, and require only a Poisson structure, \cite{KSoAir}.  In Sections~\ref{higgs} and \ref{K3} we describe the cases $X=T^*C$ and $X=$ an elliptic K3 surface.

\subsection{Deformation space of embedded curves}   \label{sec:def}
Consider a symplectic surface $(X,\Omega_X)$ together with a smooth, embedded genus $g$ curve $\Sigma\subset X$.  The deformation space $\cb$ of $\Sigma$ inside $X$ is a smooth complex analytic moduli space of dimension $g$.
The tangent space of the moduli space $\cb$ at the point $\Sigma$ is naturally identified with $H^0(\Sigma,\nu_\Sigma)$, the space of holomorphic sections of the normal bundle of $\Sigma$, which is isomorphic to $H^0(\Sigma, K_\Sigma)$, the space of holomorphic differentials on $\Sigma$ via adjunction
$$ K_\Sigma\cong K_X|_\Sigma\otimes\nu_\Sigma\cong \nu_\Sigma.
$$
This is a particular case of the more general property for a Lagrangian subspace $L$ of a symplectic vector space $W$
$$ 0\to L\to W\stackrel{\Omega_W(\cdot,\cdot)}{\to} L^*\to 0
$$
which produces a canonical isomorphism $L^*\cong W/L$. 

Over the moduli space $\cb$ is a symplectic vector bundle $\ch$ equipped with a flat connection $\nabla^{\text{GM}}$. The symplectic vector bundle $\ch$ is given by the hypercohomology
\begin{equation}  \label{cohvb}
    \ch_\Sigma=\bh^0(\Sigma, \mathrm{Cone}(d_{dR}: \co_\Sigma \rightarrow \Omega^1_\Sigma ))
\end{equation}
which is isomorphic to the first cohomology group $H^1(\Sigma;\bc)$, and $\nabla^{\text{GM}}$ is the Gauss-Manin connection.
Define the $\ch$-valued 1-form $\phi \in \Gamma(\cb,\Omega^1_\cb\otimes\ch)$ via the composition of maps
$$T_\Sigma \cb \stackrel{\cong}{\longrightarrow} H^0(\Sigma,\Omega^1_\Sigma) \rightarrow H^1(\Sigma;\bc) =\ch_\Sigma.$$
Then $\phi$ is flat with respect to $\nabla^{\mathrm{GM}}$.
\begin{lemma}
$\nabla^{\mathrm{GM}}\phi=0$.
\end{lemma}
\begin{proof}
We will show that $\phi$ is locally exact, i.e. there exists a well-defined local primitive, or equivalently that $\phi$ integrates trivially along small loops in $\cb$.  Given $[\Sigma]\in\cb$, choose a small loop $\gamma\subset\cb$ containing $[\Sigma]$.  Choose $[\alpha]\in H_1(\Sigma,\bz)$ represented by an embedded closed curve $\alpha\subset\Sigma$ and choose a family $\tilde{\alpha}$ of embedded closed curves representing the given homology cycle in each fibre.  This gives a torus $T^2\to X$ which bounds a solid torus $M^3\to X$ when $\gamma$ is chosen small enough.  Integration of $\phi$ along $\gamma$ gives an element of a fibre of $\ch$ which evaluates on $[\alpha]$ by
\[\left\langle\int_\gamma\phi,[\alpha]\right\rangle=\int_{T^2}\Omega_X=0\]
since $d\Omega_X=0$ and $T^2$ is homologically trivial.  This applies to any primitive homology class $[\alpha]$ hence
$$\int_\gamma\phi=0.$$
This is true of any small $\gamma$, so $\phi$ is locally exact hence closed (as a section of the locally trivial bundle $\ch$) i.e. flat with respect to $\nabla^{\mathrm{GM}}$.
\end{proof}

The flat connection $\nabla^{\mathrm{GM}}$ naturally defines a complex
$$\Omega^0_\cb\otimes\ch \stackrel{\nabla^{\mathrm{GM}}}{\to} \Omega^1_\cb\otimes\ch \stackrel{\nabla^{\mathrm{GM}}}{\to} \Omega^2_\cb\otimes\ch\stackrel{\nabla^{\mathrm{GM}}}{\to} ...$$
Define a local section $s\in\Gamma(U_{[\Sigma]},\ch)$, for $U_{[\Sigma]}\subset\cb$ a neighbourhood of a point $[\Sigma]\in\cb$, by $\nabla^{\text GM}s=-\phi$.  The solution to this equation is a cohomology class $s([\Sigma'])\in \ch_\Sigma'$ for each $[\Sigma']\in U_{[\Sigma]}$ well-defined up to addition of a constant independent of $[\Sigma']$. To remove the constant define $[\theta]:U_{[\Sigma]}\to \ch_\Sigma$ by
\begin{equation}  \label{cohtheta}
    [\theta]([\Sigma']):=s([\Sigma])-s([\Sigma'])\in\bc^{2g}\cong \ch_\Sigma
\end{equation}
which is a well-defined map from the open set $U_{[\Sigma]}\subset\cb$ to $\bc^{2g}$.
The isomorphism of cohomology with $\bc^{2g}$ uses a choice of Torelli basis. Strictly, in \eqref{cohtheta}, $s([\Sigma'])$ has been parallel transported from $\ch_{\Sigma'}$ to $\ch_\Sigma$ via the Gauss-Manin connection.  By definition, the covariant derivative of $[\theta]$ is given by
$$\nabla^{\text GM}_\eta[\theta]=\eta
$$
for any $\eta\in H^0(\Sigma, K_{\Sigma})\cong T_{[\Sigma]}\cb$.  The linearisation $\nabla^{\text GM}[\theta]:T_{[\Sigma']}\cb\to \ch_{\Sigma'}$ has image given by the parallel transport of the Lagrangian subspace $H^0(\Sigma, K_{\Sigma'})\subset \ch_{\Sigma'}$ so $[\theta]$ defines a local Lagrangian embedding of $U_{[\Sigma]}]\subset\cb$ into $\ch_\Sigma$:
\begin{equation}  \label{Bemb}
    U_{[\Sigma]}\stackrel{\text{Lag.}}{\longhookrightarrow}\ch_\Sigma\cong\bc^{2g}.
\end{equation}

\begin{remark}
It is important to note that $[\theta]$ is related to, but not equal to, the cohomology class of the tautological 1-form $vdu|_{\Sigma'}$ in the case $X=T^*C$.  It is given by
$$[\theta]([\Sigma'])=\left[vdu|_{\Sigma}\right]-\left[vdu|_{\Sigma'}\right]\in\bc^{2g}.$$
As mentioned above, this difference uses parallel transport by the Gauss-Manin connection.  
We will see later that there exists a  meromorphic differential $\theta$ which is defined only in a formal neighbourhood of $[\Sigma]\in\cb$ with cohomology class given by an analytic expansion of $[\theta]$.  The Gauss-Manin connection lifts to a connection with well-defined parallel transport  on any formal neighbourhood, but only partially defined on $\bg$. 
\end{remark}

Using the choice of $a$-cycles on each $\Sigma'$, $[\theta]$ defines coordinates on $U_{[\Sigma]}\subset\cb$ by
\begin{equation}  \label{zcoords}
    z^i([\Sigma'])=\oint_{a_i}[\theta]([\Sigma']),\quad i=1,...,g.
\end{equation}
The coordinates satisfy $z^i([\Sigma])=0$, and coordinates defined with respect to any nearby point are related via a constant shift $z^i\mapsto z^i+z^i_0$.


The $b$-cycles on each $\Sigma'$ give rise to functions $\w_i(z^1,...,z^g)$ defined by $\w_i=\oint_{b_i}[\theta]([\Sigma'])$ for $i=1,...,g$.  Their derivatives satisfy
$$\frac{\partial \w_i}{\partial z^j}=\frac{\partial}{\partial z^j}\oint_{b_i}[\theta]=\oint_{b_i}\nabla^{\mathrm{GM}}_{\hspace{-1mm}\frac{\partial}{\partial z^j}}[\theta]=\oint_{b_i}\omega_j=\tau_{ij}
$$
where the second equality uses the definition of the Gauss-Manin connection.  By the Riemann bilinear relations $\tau_{ij}$ is symmetric, hence there exists a function, known as the {\em prepotential}, 
\[F_0:U\to\bc\] 
satisfying
$$\w_i=\frac{\partial F_0}{\partial z^i},\quad i=1,...,g
$$
hence also
$\frac{\partial^2 F_0}{\partial z^i\partial z^j}=\tau_{ij}$ and $\frac{\partial^3 F_0}{\partial z^i\partial z^j\partial z^k}=c_{ijk},$
which defines the tensor $\bar{A}_\Sigma\in\overline{V}_\Sigma\otimes\overline{V}_\Sigma\otimes\overline{V}_\Sigma$.

\subsection{Foliations and topological recursion}   \label{foliation}

Equip the symplectic surface $(X,\Omega_X)$ with a holomorphic Lagrangian foliation $\cF$. Let $\Sigma\subset(X,\Omega_X,\cf)$ be a compact curve embedded inside a symplectic surface $(X,\Omega_X)$ with Lagrangian foliation $\cf$.  We require that $\Sigma$ is tangent to $\cf$ at finitely many points $R \subset\Sigma$ and the tangencies are simple.  
\begin{example}
A typical example is the cotangent space $X=T^*C$ of a compact curve $C$. The symplectic surface $X$ is foliated by fibres of the projection $\pi: X \rightarrow C$.  For an embedded compact curve $\Sigma\subset X$ the set $R$ is the set of ramification points of the morphism $\pi|_{\Sigma}$, and $\Sigma$ is chosen to have simple ramification points.
\end{example}

\begin{definition}  \label{FDcoord}
Define FD (foliation-Darboux) local coordinates $u,v$ on $(X,\Omega_X,\cf)$ to be Darboux coordinates, i.e. $du\wedge dv=\Omega_X$ that define the leaves of the foliation via $u=$ constant.
\end{definition}
For each point in $X$, there exists a neighbourhood with FD local coordinates.  They are unique up to the symplectic change of coordinates which preserves the foliation
\begin{equation}  \label{xyamb}
    (u,v)\mapsto(f(u),\frac{v}{f'(u)}+g(u)),
\end{equation}
for $f'(u)\neq 0$ in the neighbourhood of $X$.

The data $\Sigma\subset(X,\Omega_X,\cf)$ gives rise to a spectral curve which is used to define topological recursion.  We begin with a definition of topological recursion following Eynard and Orantin \cite{eynard_orantin}.  Topological recursion arose out of the study of the free energy of matrix models \cite{CEyHer}.

\begin{itemize}
\item {\bf Spectral curve.} A spectral curve $(\Sigma,u,v,B)$ consists of a compact Riemann surface $\Sigma$ equipped with two meromorphic functions $u$ and $v$ defined on $\Sigma$ and a symmetric bidifferential $B$ defined on $\Sigma\times\Sigma$.  We assume that each zero of $du$ is simple and does not coincide with a zero of $dv$. Topological recursion produces symmetric tensor products of meromorphic differentials $\omega_{h,n}$ on $\Sigma^n$ for $h \geq 0$ and $n \geq 1$ which we call correlators. 
\item {\bf Bergman kernel.}
A good choice of bidifferential $B$ in the spectral curve is the {\em Bergman kernel} which is a canonical normalised symmetric bidifferential $B(p,p')$ associated to a compact Riemann surface equipped with a choice of $a$-cycles $\{a_i\}_{i=1,...,g}\subset\Sigma$.  It is normalised by $\int_{p\in a_i}B(p,p')=0$, $i=1,...,g$.  In a local coordinate $z$ on $\Sigma$ it is given by
\begin{equation}  \label{bidiff}
B(p,p')=\frac{dz(p)dz(p')}{(z(p)-z(p'))^2}+\text{holomorphic in }(z(p),z(p')).
\end{equation}
It generalises the Cauchy kernel since it satisfies $\displaystyle df(p)=\Res_{p'=p}f(p')B(p,p')$, for all meromorphic $f$.
\item {\bf Recursion kernel.} Define a kernel in a neighbourhood of any $\alpha\in\Sigma$, i.e. $du(\alpha)=0$, by
\begin{equation}  \label{kernel}
K(p_1,p) = -\frac12\frac{\int_{\hat{p}}^p B(p_1, \,\cdot\,)}{(v(p)-v(\hat{p})) \, du(p)}.
\end{equation}
where $p \mapsto \hat{p}$ denotes the holomorphic involution defined locally at the ramification point $\alpha\in R$ satisfying $u(\hat{p}) = u(p)$ and $\hat{p} \neq p$.
\item {\bf Recursion.} The correlators $\omega_{h,n}$ are defined by 
\[\omega_{0,2}(p_1,p_2)=B(p_1,p_2)
\]
and for $2h-2+n>0$ recursively via:
\begin{align}   \label{rec}
\omega_{h,n}(p_1, p_S) = \sum_{du(\alpha) = 0} \mathop{\mathrm{Res}}_{p=\alpha} \, & K(p_1, p) \Bigg[ \omega_{h-1,n+1}(p, \sigma_\alpha(p), p_S) \\
&+\hspace{-2mm} \mathop{\sum_{h_1+h_2=h}}_{I \sqcup J = S} \omega_{h_1,|I|+1}(p, p_I) \, \omega_{h_2,|J|+1}(\sigma_\alpha(p), p_J) \Bigg].
\nonumber
\end{align}
Here, we use the notation $S = \{2, 3, \ldots, n\}$ and $p_I = \{p_{i_1}, p_{i_2}, \ldots, p_{i_k}\}$ for $I = \{i_1, i_2, \ldots, i_k\}$. The outer summation is over the zeroes of $du$. 
\item {\bf Structure of correlators.} The correlators $\omega_{h,n}(p_1,..., p_n)$ are tensor products of meromorphic differentials, symmetric in $p_i$, with zero residue poles at $p_i=\alpha$ for any zero $\alpha$ of $du$, and holomorphic outside the set defined by $du=0$. They inherit from $B(p,p')$ the property $\oint_{p_i\in a_k}\omega_{h,n}(p_1,\cdots,p_n) = 0$.
\item {\bf Dilaton equation.}   
The differential $vdu$ is locally exact on $\Sigma$, and we define $\psi$ to be a local primitive, i.e. $d\psi=vdu$.  The {\em dilaton equation}, proven in \cite{eynard_orantin} is:
\begin{equation}  \label{dilaton}
\sum_{du(\alpha) = 0}\mathop{\mathrm{Res}}_{p_{n+1}=\alpha}\psi(p_{n+1})\omega_{h,n+1}(p_1,...,p_{n+1})=(2h-2+n)\omega_{h,n}(p_1,...,p_n).
\end{equation}
Since $\omega_{h,n+1}$ has zero residue at each $\alpha$, the left hand side of the dilaton equation is independent of the choice of primitive $\psi$.  The dilaton equation leads to the definition of the correlators for $n=0$ and $h\geq 2$.
\begin{equation}  \label{hatf}
F_h:=\frac{1}{2h-2}\sum_{du(\alpha) = 0}\mathop{\text{Res}}_{p=\alpha}\psi(p)\omega_{h,1}(p),\quad h\geq 2.
\end{equation}
These are called {\em symplectic invariants} in \cite{eynard_orantin} (which uses $F_h$ that differs by a negative sign from \eqref{hatf}).
\item {\bf Local spectral curve.} The recursion depends only a neighbourhood of the zeros of $du$, hence $u$, $v$ and $B$ need only be defined locally in this neighbourhood.  In this case $(\Sigma,u,v,B)$ is said to be a {\em local} spectral curve.
\end{itemize} 
In this paper $\omega_{0,1}$ are not defined, or equivalently zero.  In some conventions $\omega_{0,1}$ is defined to coincide with $vdu$.

The recursive procedure of topological recursion \eqref{rec} can be formulated and generalised \cite{ABCOABCD,KSoAir} in terms of the tensors $A$, $B$ and $C$ from the Airy structure defined in Section~\ref{airystr}. 

\subsubsection{Correlators of \texorpdfstring{\boldmath$\Sigma\subset(X,\Omega_X,\cf)$}{Sigma subset (X,Omega,F)}.}   \label{corspec}

Given a compact curve embedded inside a symplectic surface with Lagrangian foliation $\Sigma\subset(X,\Omega_X,\cf)$, choose a collection of $a$-cycles on $\Sigma\subset X$. This choice defines an associated Bergman kernel $B$ normalised over the $a$-cycles, and together with a choice of FD local coordinates $u$ and $v$ (see Definition~\ref{FDcoord}), they define a local spectral curve:
$$\Sigma\subset(X,\Omega_X,\cf)\leadsto(\Sigma,u,v,B).
$$
Apply topological recursion, defined via \eqref{kernel} and \eqref{rec}, to this local spectral curve
to produces correlators $\omega_{h,n}$ which are tensor products of meromorphic differentials on $\Sigma^n$ with poles precisely at $R\subset\Sigma$.
The spectral curve $(\Sigma,u,v,B)$ depends on a choice of $(u,v)$. The correlators are independent of the ambiguity \eqref{xyamb} since $(u,v)$ enters the recursion via the kernel $K(p_1,p)$ as 
\[(v(p)-v(\hat{p}))du(p)=\left(\frac{v(p)}{f'(u(p))}+g(u(p))-\frac{v(\hat{p})}{f'(u(\hat{p}))}-g(u(\hat{p}))\right)  df(p)\]  
and the involution $p\mapsto\hat{p}$ depends only on the foliation.  The dilaton equation is also independent of the ambiguity \eqref{xyamb} since $\psi\mapsto \psi+\xi(u)$ which adjusts the left hand side of \eqref{dilaton} by a sum of residues of a holomorphic function in $u$ times the correlator.  These residues vanish, i.e. for holomorphic $\xi(u)$ defined in a neighbourhood of $\alpha$, $\displaystyle\mathop{\mathrm{Res}}_{p=\alpha}\xi(u(p))\omega_{h,n}|_{p_{n}=p}=0$ since the principal part of $\omega_{h,n}$ at $\alpha$ is skew-invariant under the involution $p\mapsto\hat{p}$ and it is still skew-invariant after multiplication by an invariant function.  The residue comes from the invariant part.  In particular the functions $F_h$ are well-defined for $h\geq 2$ since they do not change under \eqref{xyamb}.  In the case $X=T^*X$, this generalisation of a spectral curve was studied in \cite{DMuQua}.

Since the simple tangency condition on a curve $\Sigma\subset (X,\Omega_X,\cf)$ is an open condition and a choice of $a$-cycles on $\Sigma$ is a discrete choice, we can choose an open neighbourhood $U_{[\Sigma]}$ of $[\Sigma]\in\cb$ consisting of nearby embedded $\Sigma'\subset (X,\Omega_X,\cf)$ satisfying the simple tangency condition and with a given choice of $a$-cycles.  Thus the correlators $\omega_{h,n}$ are well-defined on each nearby $\Sigma'$ and $F_h$ defines a function on a neighbourhood $U_{[\Sigma]}$ of $[\Sigma]\in\cb$ for each $h\geq2$.  

One main motivation of \cite{KSoAir}, is to use the functions $F_h$ to produce a cyclic vector for the deformation quantisation of $\cb\subset H^1(\Sigma;\bc)$ by:
    \[\exp\left(\frac{F_0}{\hbar}+F_1+\hbar F_2+...+\hbar^{g-1}F_g+...\right)
    \]
This is annihilated up to $O(\hbar)$ by a quantisation of the local defining equations for $\cb\subset H^1(\Sigma;\bc)$:
    \[ -\hbar\frac{\partial}{\partial z^i}+\w_i(z^1,...,z^g).\] 
Just as $F_0$ can be calculated independently of the choice of foliation, the deformation quantisation suggests that there may be a way one could define the $F_h$ independently of the choice of foliation. This might allow $F_h$ to be constructed via topological recursion using any local Darboux coordinates $(u,v)$ of $X$ leading to symplectic invariance of $F_h$.   
    


\subsection{Cohomological field theories}  \label{cohft}
A {\em cohomological field theory} (CohFT) is a pair $(H,\langle.,.\rangle)$ consisting of a finite-dimensional complex vector space $H\cong\bc^R$ equipped with a non-degenerate symmetric bilinear pairing $\langle.,.\rangle$ and a sequence of $S_n$-equivariant maps 
\[ \Omega_{h,n}:H^{\otimes n}\to H^*(\overline{\modm}_{h,n};\bc).\]
The maps $\Omega_{h,n}$ satisfy natural compatibility conditions with respect to restriction to lower dimensional strata in $\overline{\modm}_{h,n}$ built out of $\overline{\modm}_{h',n'}$---see \cite{KMaGro}.  It is semisimple if $H$ is semisimple with respect to a product on $H$ induced from $\Omega_{0,3}$ and $\langle.,.\rangle$.

A relationship between semisimple cohomological field theories and topological recursion was proven in \cite{DOSSIde}.  The correlators $\omega_{h,n}$ of $(\Sigma,u,v,B)$ are polynomial in a basis of differentials $\{\xi^\alpha_k\mid\alpha\in R\subset\Sigma,\ k\in\bn\}$  constructed out of the locally defined function $u$ on $\Sigma$ and the Bergman kernel $B$---see \cite{EynInv,EynTop}.  Define the topological recursion partition function of the spectral curve $S=(\Sigma,B,u,v)$ by
$$Z^S(\hbar,\{\xi^\alpha_k\})=\exp\left(\sum_{h,n}\frac{\hbar^{h-1}}{n!}\omega_{h,n}(\{\xi^\alpha_k\})\right).
$$
It was proven in \cite{DOSSIde} that, under assumptions on the spectral curve, $Z^S$ coincides with the partition function of a semisimple CohFT which stores intersection numbers of all $\Omega_{h,n}$ with the tautological psi classes.    
Furthermore, this decomposition coincides with a decomposition of Givental \cite{GivGro} for partition functions arising out of semisimple cohomological field theories. The assumptions on the spectral curve in \cite{DOSSIde} were lifted in \cite{CNoTop} to allow any compact curve $\Sigma\subset X$.   

Given a curve $\Sigma$ inside a foliated symplectic surface $X$ and $\alpha\in R\subset\Sigma$, choose local coordinates $(u_\alpha,v_\alpha)$ for $X$ in a neighbourhood of $\alpha$ as follows.
   \begin{definition}  \label{specoord}
   Given $(X,\Omega_X,\cf)$ and $\alpha\in R\subset\Sigma\subset X$, define local coordinates $(u_\alpha,v_\alpha)$ in a neighbourhood \(U_\alpha \subset X\) of $\alpha$ satisfying:
   \begin{itemize}
       \item  $du_\alpha\wedge dv_\alpha=\Omega_X$; 
       \item $\{u_\alpha=$ constant $\}$ defines the leaves of the foliation $\cf$;
       \item $(u_\alpha,v_\alpha)|_\alpha=(0,0)$;
       \item $u_\alpha-v_\alpha^2=0$ locally defines $\Sigma$.
   \end{itemize}
   \end{definition}
   The first two properties define FD coordinates---see Definition~\ref{FDcoord}. Via the change of coordinates given in \eqref{xyamb} arbitrary FD coordinates can be transformed to satisfy the remaining two properties. The four properties uniquely determine the coordinates up to
   \[ (u_\alpha,v_\alpha)\mapsto (\zeta^2u_\alpha,\zeta v_\alpha),\quad \zeta^3=1. \]   
The locally defined function $u$ restricts to each $\Sigma'$ for $[\Sigma']\in U_\Sigma$ and we denote its critical value by $u(\alpha')=\lambda_\alpha([\Sigma'])$. The set of critical values $\{\lambda_\alpha(z^1,...,z^g)\mid\alpha\in R\}$ defines a map
\begin{equation}  \label{coordtocritval}
    \Lambda:U_\Sigma\to\bc^R.
\end{equation}
The linearisation, described explicitly in \eqref{Lambdalin} in Appendix~\ref{variation},
\[D\Lambda:H^0(\Sigma,K_\Sigma)\to\bc^R
\]
composed with the CohFT induces linear symmetric maps:
\[H^0(\Sigma,K_\Sigma)^{\otimes n}\to H^*(\overline{\modm}_{h,n};\bc).
\]
This is no longer a CohFT because the pairing on $H^0(\Sigma,K_\Sigma)$ given by
\[\langle\eta_1,\eta_2\rangle:=\sum_{\alpha\in R}\Res_\alpha\frac{\eta_1(p)\eta_2(p)}{du_\alpha(p)}
\]
for $\eta_1, \eta_2\in H^0(\Sigma ,K_\Sigma )$
is not necessarily non-degenerate.  For example, when $X=T^*C$, the tautological 1-form $vdu|_\Sigma\in H^0(\Sigma,K_\Sigma)$ pairs trivially with any $\eta\in H^0(\Sigma ,K_\Sigma )$.

The classes $\Omega_{h,n}(\eta_1,\otimes...\otimes\eta_n)\in H^*(\overline{\modm}_{h,n};\bc)$ of a CohFT consist of terms in all degrees.  Among these, the term of degree $3h-3+n$ is known as the {\em primary} class and measured by $\int_{\overline{\modm}_{h,n}}\Omega_{h,n}(\eta_1,\otimes...\otimes\eta_n)$.  The correlator $\omega_{h,n}\in V_\Sigma^{\otimes n}$, which stores intersection numbers of the tautological psi classes with the image of $\Omega_{h,n}$, also defines a linear map \[\omega_{h,n}:H^0(\Sigma,K_\Sigma)^{\otimes n}\to\bc\] 
via the natural pairing of $V_\Sigma$ and $H^0(\Sigma,K_\Sigma)$.  It would be interesting to understand how to relate this to a primary part. 
Note that primary part of $\Omega_{h,n}$ should not be confused with the topological part, underlying any CohFT, given by the projection of $\Omega_{h,n}$ to $H^0(\overline{\modm}_{h,n}, \mathbb{C})\cong\bc$.  The projection to $H^0(\overline{\modm}_{h,n},\mathbb{C})$ defines a two-dimensional topological field theory on $(H,\langle.,.\rangle)$ which is a sequence of $S_n$-equivariant maps 
\[ \Omega^0_{h,n}:H^{\otimes n}\to\bc\]
satisfying compatibility conditions that are equivalent to composition of multilinear maps.

A CohFT on $H\cong\bc^R$ is equivalent to a geometric structure on $H$ given by a flat metric, a product on the tangent space and further structure, known as a Frobenius manifold, \cite{DubGeo}.  The manifold $\bc^R$ parametrises a family of more general deformations of $\Sigma\subset X$ than those that embed into $X$.  The family of curves gives rise to Dubrovin's superpotential associated to a 
semisimple Frobenius manifold, which is related directly to topological recursion in \cite{DNOPSSup}.  Hence the $g$-dimensional space $\cb$ of deformations of the spectral curve inside $X$ maps to an $|R|$-dimensional Frobenius manifold.

\subsection{Deformation space associated with Higgs bundles}   \label{higgs}

    A particularly interesting class of examples of a deformation spaces of a curve inside a foliated symplectic surface arises from the geometry of Higgs bundles defined by Hitchin in \cite{HitSel}.  
\begin{definition}  \label{defHiggs}
A Higgs bundle over a compact Riemann surface $C$ is a pair $(E,\phi)$ where $E$ is a rank $N$ holomorphic vector bundle over $C$ and $\phi\in H^0(C,\text{End}(E)\otimes K_C)$.
\end{definition}  
Associated to the pair $(E,\phi)$ is its spectral curve
$$\Sigma=\{\det(\phi-\lambda I)=0\}\subset T^*C,
$$
which has equation $0=(-1)^N\det{(\phi-\lambda I)}=\lambda^N+a_1\lambda^{N-1}+...+a_N$ where $a_k\in H^0(C,K_C^{\otimes k})$.  If the spectral curve is irreducible then the pair $(E,\phi)$ is {\em stable} meaning that for any $\phi$-invariant subbundle $F\subset E$, i.e. $\phi(F)\subset F\otimes K_\Sigma$, we have $\frac{c_1(F)}{\text{rank\ }F}<\frac{c_1(E)}{\text{rank\ }E}$.

The spectral curve associated to a pair defines a map from the moduli space $\cm=\cm_{N,d}$ of stable Higgs bundles of rank $N$ and degree $d$ on a compact Riemann surface $C$ of genus $g_C>1$, 
$$f:\cm \to \cb.$$
Here $\cb$ is the space of (possibly singular) spectral curves which can be identified with the following space:
    $$\cb=\bigoplus_{j=1}^N H^0(C,K_C^{\otimes j}).$$
    Fibres of $f$ are complex tori and they are singular in general.
    Let $f:\cm^{reg} \to \cb^{reg}$ be the restriction of $f$ to the open subset $\cm^{reg} \subset \cm$ consisting of smooth fibres.
    For any point $[\Sigma] \in \cb^{reg}$ the associated spectral curve $\Sigma$ is an irreducible curve of genus $g=N^2(g_C-1)+1$ (which is calculated via $\dim\cb=1+\sum_{j=1}^N(g_C-1)(2j-1)$).
    The deformation space of $\Sigma \subset T^* C$ coincides with $\cb^{reg}$.
    The natural projection $\pi: \Sigma \rightarrow C$ is a degree $N$ map.  The foliation is given by fibres of the projection map $\pi$. We consider only $\Sigma$ such that the morphism $\pi$ has only double ramification points.   
    
    Fibres over $\cb^{reg}$ are naturally identified with Jacobians of the spectral curves $\Sigma$ for $[\Sigma] \in \cb$ which is defined inside the cotangent bundle of the Riemann surface $C$. The tangent space of a fibre is naturally identified with $H^1(\Sigma,\co_\Sigma)$. The moduli space $\cm^{reg}$ is symplectic and the symplectic form produces a non-degenerate pairing between the tangent space of the base and the tangent base of the fibre
    \[H^0(\Sigma,K_\Sigma)\otimes H^1(\Sigma,\co_\Sigma)\to\bc
    \]
    which coincides with Serre duality.  The base space $\cb^{reg}$ parametrises embedded Lagrangian Jacobians in $\cm^{reg}$ and  embedded  curves $\Sigma\subset X$, which are automatically Lagrangian.  It is proven in  \cite{HitMod} that the deformation space of a compact holomorphic Lagrangian in a holomorphic symplectic K\"ahler manfold naturally has a special K\"ahler structure.  Hence there are two natural special K\"ahler structures defined on $\cb^{reg}$.  It is proven in \cite{BHuSpe} that the special K\"ahler structures coincide---see also \cite{HitInt}.


    Denote by $vdu$ the tautological 1-form on the cotangent bundle of $C$.  The pair of holomorphic coordinate systems $(\zeta^1, \dots, \zeta^g)$ and $(\eta_1, \dots, \eta_g)$ on $U_\Sigma\subset\cb^{reg}$ is obtained by integrating the 1-form $vdu$ over $a$-cycles and $b$-cycles of the spectral curve. More precisely, 
    $$\zeta^i=\int_{a_i}vdu, \qquad \eta_i=\int_{b_i}vdu.$$
    Given normalised holomorphic differentials $\omega_i$, $i=1,...,g$ on $\Sigma$, $vdu- \zeta^i\omega_i$ is holomorphic with zero $a$-periods so it vanishes, hence:
$$vdu= \zeta^i\omega_i.$$
    The tautological 1-form gives a canonical primitive of $\Omega_X$, which does not exist for more general symplectic $X$ so we instead use the coordinates $z^i$ defined in \eqref{zcoords} and the related coordinates $\w_i$.  The relation between these coordinates is as follows.  For $[\Sigma']\in U_\Sigma$
$$ z^i([\Sigma'])=\zeta^i([\Sigma'])-\zeta^i([\Sigma]),\qquad \w_i([\Sigma'])=\eta^i([\Sigma'])-\eta^i([\Sigma]).
$$
In particular, $z^i(\Sigma)=0=\w_i(\Sigma)$.  Note that we still have
$\tau_{ij}=\frac{\partial}{\partial z^i}\w_j$
since $\tau_{ij}=\frac{\partial}{\partial \zeta^i}\eta_j=\frac{\partial}{\partial z^i}(\w_j+\text{constant})=\frac{\partial}{\partial z^i}\w_j$.
    
    The action of $\bc^*$ on fibres of $T^*C$ induces an action on $\cb$ which preserves the conformal type of the spectral curve hence also $\tau_{ij}$ is preserved.  Under this action $vdu\mapsto\lambda vdu$ for $\lambda\in\bc^*$ hence $z^i\mapsto\lambda z^i$ and $\w_i\mapsto\lambda \w_i$. 
    We have
    \[F_0=\frac12 z^i\w_i\]
    since
    \begin{align*}
    \frac{\partial}{\partial z^j}\frac12 z^i\w_i&=\frac12\w_j+\frac12 z^i\frac{\partial}{\partial z^j}\w_i
    =\frac12\w_j+\frac12 z^i\tau_{ij}=\frac12\w_j+\frac12 z^i\tau_{ji}\\
    &=\frac12\w_j+\frac12 z^i\frac{\partial}{\partial z^i}\w_j=\frac12\w_j+\frac12\w_j=\w_j,
    \end{align*}
    where the second last equality used the fact that $\w_i$ is homogeneous of degree one under the $\bc^*$ action which is generated by $z^i\frac{\partial}{\partial z^i}$.  For more general symplectic $X\neq T^*C$, $F_0$ does not have the same simple formula.

    For $h\geq 2$ there is a similar formula for $F_h$.    
    \begin{align*}
    F_h&=\frac{1}{2h-2}\sum_\alpha\Res_{p=\alpha}\psi(p)\omega_{h,1}(p)\\
    &=\frac{1}{2h-2}\sum_{i=1}^g\oint_{a_i}\omega_{h,1}(p)\oint_{b_i}vdu(p)-\oint_{b_i}\omega_{h,1}(p)\oint_{a_i}vdu(p)\\
    &=\frac{1}{2-2h}\oint_{b_i}\omega_{h,1}(p)z^i
    \end{align*}
    where we sum over indices $i=1,...,g$ in the last expression, $\psi(p)$ is a primitive of the restriction of the tautological 1-form $vdu(p)$ on $\Sigma-\{a_i,b_i\}$ and we have used the Riemann bilinear relations.  Note also that $F_h$ is homogeneous of degree $2-2h$ which follows from topological recursion since inductively the recursion gives $\omega_{h,n}\mapsto \lambda^{2-2h-n}\omega_{h,n}$ under the $\bc^*$ action.

\subsubsection{Rank one case} 
    A rather trivial example is the rank one case which gives the deformation space of the zero section of a cotangent bundle $\Sigma\subset T^*\Sigma$.  The deformation space is $\cb=H^0(\Sigma,K_\Sigma)$ since any deformation of the zero section remains a section.  The vector space $\cb$ is isomorphic to its tangent space $T_{[\Sigma]}\cb=\cb$.  We have $\phi \in \Gamma(\cb,\Omega^1_\cb\otimes\ch)$ defined by $\phi(\eta)=[\eta]\in H^1(\Sigma;\bc)$ for any $\eta\in H^0(\Sigma,K_\Sigma)$.  The $a$-periods of $vdu$ define coordinates $z^i$ on $\cb$.    The sum $z^i[\omega_i]$, with respect to the basis of normalised holomorphic differentials $\omega_i$, $i=1,...,g$, represents the general point in $\cb$ and also the restriction of the tautological 1-form.  The Lagrangian embedding $\cb\to\bc^{2g}$ is defined globally and is simply the linear embedding $H^0(\Sigma,K_\Sigma)\to H^1(\Sigma;\bc)$.  The prepotential is $F_0=z^i\w_i=z^iz^j\tau_{ij}$ where $\tau_{ij}$ is constant on $\cb$, and the deformation tensor vanishes identically: $0=\bar{A}_\Sigma\in\overline{V}_\Sigma\otimes\overline{V}_\Sigma\otimes\overline{V}_\Sigma$.  In order to agree with the more general construction of coordinates $z^i$ for a symplectic surface, we would shift the coordinates by a constant $z^i-z^i([\Sigma])$. 

Note that for smooth symplectic structures, by Weinstein's theorem \cite{WeiSym} the neighbourhood of any Lagrangian submanifold is symplectomorphic to a neighbourhood of the zero section of the Lagrangian submanifold in its cotangent bundle equipped with its canonical symplectic structure. Unlike in the smooth category, a neighbourhood of a complex submanifold is not necessarily biholomorphically equivalent to a neighbourhood of a complex submanifold in the total space of its normal bundle.   
If a local holomorphic symplectomorphism exists between neighbourhoods of $\Sigma\subset X$ and $\Sigma\subset T^*\Sigma$ then the prepotential of the former must coincide with the prepotential in the rank one case above. 

\subsection{K3 surfaces}  \label{K3}

A rich class of examples of foliated symplectic surfaces arise from elliptic fibrations of K3 surfaces. Elliptic K3 surfaces form a dense codimension one subset of the moduli space of complex K3 surfaces. 
Recall that a K3 surface is called elliptic when there is a surjective morphism $\pi: X \rightarrow \bp^1$ whose generic fibre is a smooth curve of genus one.
Such morphisms are always flat and therefore all fibres have arithmetic genus one. The foliation $\cf$ on any such surface $X$ is defined by the fibres of the elliptic fibration.  The foliation is singular at a finite set of points, which can be avoided by a generic spectral curve inside the K3 surface.

The simplest class of K3 elliptic surfaces are obtained from Kummer surfaces of the form $X=E_1 \times E_2$, where $E_1$ and $E_2$ are elliptic curves. Two special fibrations are obtained via the projections of the surface $X$ to the quotients $E_i/\iota \cong \bp^1$, where $\iota$ is the elliptic involution on the curve. 

Elliptic surfaces $X\to\bp^1$ with a section are described by their Weierstrass form:
\[\fy^2=\fx^3+f(\fz)\fx+g(\fz)\]
where $\fx,\fy,\fz$ are local coordinates and $f(\fz), g(\fz)$ are polynomials of $\deg f=8$, $\deg g=12$. We equip the surface with the symplectic form 
$$\omega=\frac{d\fx \wedge d\fz}{\fy}.$$  
This equation defines an affine surface in $\bc^3$ with compactification $X$.

For a given polynomial 
\[p(z)=u_0+u_1z+ \dots + u_m z^m\]
we obtain a hyperelliptic curve $\Sigma\subset X$ defined via the equation $x=p(z)$, or equivalently
$$y^2=p(z)^3+f(z)p(z)+g(z).$$
The deformation space $\cb$ is parametrised by $\{u_0,...,u_m\}$.
An example of such a family of hyperelliptic curves can be found in \cite{takasaki2001hyperelliptic}.
 A choice of $a$-cycles on $\Sigma$ together with the foliation defines a spectral curve as in Section~\ref{corspec} and Theorem~\ref{thetatheorem} applies. 

Another family of examples of elliptic K3 surfaces arise from quartics in $\bp^3$, such as  the Fermat quartic:
$$X=\{z_0^4+z_1^4+z_2^4+z_3^4=0\}\subset\bp^3.
$$
The set of hyperplanes in $\bp^3$ will be used to define both the deformation space $\cb$ of curves in $X$ and
the elliptic fibration $X\to\bp^1$ as follows.  For $H$ a hyperplane in $\bp^3$ intersecting $X$ generically, let $\Sigma=H\cap X$ be an embedded genus $3$ curve.  Its deformation space $\cb$ is an open set in the set of hyperplanes $\bp^3_{\text dual}$ in $\bp^3$.  Consider a line in $\bp^3$ that is contained in $X$, i.e. $L\subset X$.  The $\bp^1$-family of hyperplanes in $\bp^3$ that contain $L$ defines the elliptic fibration $X\to\bp^1$.  A choice of $a$ and $b$-cycles on $\Sigma$ defines correlators on $\omega_{h,n}$ on $\Sigma$.  Again, Theorem~\ref{thetatheorem} applies in this case.

We can replace the genus 3 curve $\Sigma\subset X$ in the previous example by a genus 1 curve.  Instead choose a hyperplane in $\bp^3$ intersecting $X$ non-generically.  Given two such lines $L_1,L_2\subset\bp^3$ and a hyperplane $H$ containing $L_1$ define the elliptic curve $\Sigma=H\cap X-L$ with a 1-dimensional deformation space $\cb$, and use $L_2$ to define a foliation on $X$.  In other words, given
two different elliptic fibrations, we use one for the foliation and the other for the family of embedded curves.

    \section{Airy structures}  \label{quadlag} 
    In this section we give the formulation of topological recursion due to Kontsevich and Soibelman \cite{KSoAir}.  Given a quadratic Lagrangian $\cl\subset W$ we describe the corresponding Airy structure which is a collection of tensors satisfying $A$, $B$ and $C$ on $V$ satisfying quadratic constraints.   We then define the main example given by the quadratic Lagrangian $\cl_{\text{KS}}$ defined in \eqref{quadlagINTRO} which gives rise to abstract topological recursion \eqref{abstractTR}.

\subsection{Tate spaces}
    We outline the algebraic background needed to define an \emph{Airy structure}, on an infinite dimensional \emph{Tate space} \( W \), over a field \( \mathbf{k}\) of characteristic zero with discrete topology, from \cite{KSoAir}. Tate spaces can be used as a model for an infinite dimensional symplectic (topological) vector space. 
    
    Let \(V, U\) be topological vector spaces, with \emph{discrete toplogy}, both over a field \( \mathbf{k}\) with discrete topology. Let \(*\) denote the topological dual. Recall the discrete topology defines all subsets as open sets. 

    \begin{definition}[Tate space]
    A Tate space \(W\) is the direct sum
    \[ W = V \oplus U^*.\]
    \end{definition}
    As \(U\) has discrete topology, \(U^*\) has \emph{locally linearly compact} topology. 
    
    A topological vector space \(V'\) is \emph{linearly topologised} if there is a neighbourhood basis at zero of linear subspaces, and is Hausdorff. A linear variety \(A\) is a subset of the form \(v + U'\) where \(U'\) is a linear subspace of \(V'\). \(A\) is closed if \(U'\) is closed (or respectively open). Finally a linearly topologised vector space is \emph{linearly compact} if collections of linear varieties with the finite intersection property is non empty \cite[p.~74]{lefschetz1942algebraic}.
    
    Setting \(U=V\), there is an isomorphism \(V \cong (V^*)^*\) \cite{drinfeld_tate}. With \(U=V\), this gives \(W\) the property \(W \cong W^*\), making \(W\) a strong symplectic vector space, with a polarisation given by  \(V\).
    
    An Airy structure characterises a quadratic \emph{Lagrangian} subvariety \( \mathcal{L} \) in the polarised symplectic vector space \(W\). Define \( \cl\) as the zero set of the ideal generated by a collection of quadratic polynomials. Choose \emph{Darboux} coordinates \(\{x^k\mid k\in I\} \) on \(V^* \cong L=T_0\cl\) indexed by a set \( I\subseteq \mathbb{N}\), and note \( x^k \in (V^*)^* \cong V\), together with coordinates \(\{y_k\mid k\in I\} \) on \(V\), so that \(y_k \in  V^* \cong L\). The coordinates \(x^k\) and \(y_k\) can also be treated as formal variables in a coordinate ring.  For the main infinite-dimensional example $W$ in this paper, we will choose a particular set of Darboux coordinates, given in Definition~\ref{darbcoord}. 
    

    When \(W\) is infinite dimensional and \(I = \mathbb{N}\), we construct the coordinate ring \( \mathbf{k}[W]\) via the symmetric algebra:
    \[ \mathrm{S}(W^*) = \bigoplus \mathrm{S}^k( V^* \oplus V) = \bigoplus_k  \bigoplus_{\lambda \vdash k} \mathrm{S}^{\lambda}(V) \otimes \mathrm{S}^{\lambda}(V^*)  \cong \mathbf{k}[V] \otimes \mathbf{k}[V^*] \cong \mathbf{k}[W],\] 
    where \(\otimes\) is algebraic tensor product.
    Elements of \( \mathbf{k}[W]\) are formal combinations of variables \(x^{\sbt}\) and \( y_{\sbt}\), $x^k\neq 0$ only for a $k$ in a finite number of $k\in I$, \( x^{\sbt}\) terms are bounded in degree, and \( y_{\sbt}\) terms are bounded in degree. 
    
    Additionally, there is a natural isomorphism between completed tensor products, and a ring of formal power series in infinite variables, 
    \[\widehat{\mathrm{S}}(W^*)  = \mathbf{k}\llbracket W \rrbracket.  \]
    Where \( \mathbf{k}\llbracket W \rrbracket \) is given by completion at the maximal ideal \( \langle x^{\sbt}, y_{\sbt} \rangle \), allowing for formal sums of unbounded degree in \( x^{\sbt}\) and \( y_{\sbt}\).
    
    The symplectic structure on the vector space \(W\) corresponds naturally to a \emph{Poisson bracket} on \( \mathbf{k}[W]\) and on the completion \(\mathbf{k}\llbracket W \rrbracket\). The Poisson bracket is a map \[ \{ {\sbt}, {\sbt}\}: \mathbf{k}[ W]  \times \mathbf{k}[ W ] \rightarrow \mathbf{k}[ W ]\] defined by the coordinates, \[ \{ y_j,x^i\} := \delta^i_j,\quad \{x^i,x^j\} := 0,\quad \{ y_i, y_j \}: = 0\] and extending to polynomials and formal series via the Leibniz rule. This gives \( k[W]\) the structure of a Lie algebra.
    
    \begin{example}
    Let $(\Sigma,R)$ be a curve equipped with a divisor $R\subset\Sigma$. Equip $\Sigma$ with a choice of $a$-cycles in $H^1(\Sigma;\bz)$.  The main example we consider in this paper is 
    \(V_\Sigma \) 
    defined in \eqref{vsigma} to be the vector space of residueless global meromorphic differentials on $\Sigma$, holomorphic on $\Sigma-R$ with zero \(a\)-periods.  It is equipped with the discrete topology.
    \end{example}
    
    
\subsection{Airy structures}   \label{airystr}
    
    Consider a quadratic Lagrangian \(\mathcal{L}\) in \(W\) defined by a (possibly infinite) collection of quadratic polynomials, \( H_i \in \mathbf{k}[ W ] \) given by
    \[ H_i=-y_i+a_{ijk} \, x^j x^k+ 2 b_{ij }^{k } \, x^j y_k +c_{i}^{ jk}\,y_j y_k, \quad i,j,k \in I \subseteq \mathbb{N}\]
    where we sum over repeated indices.
    Linearising \(H_i\) defines the tangent space at \(0\) by \( T_0 \mathcal{L} =\{y_i=0\}\).
     
    With respect to a polarisation, i.e. a choice of \(V\subset W\) such that $W\cong V\oplus V^*$, the coefficients  naturally form tensors:
    \begin{alignat}{3}
    \label{airytensor}
    A &=(a_{ijk})  &&\in V\otimes V\otimes V,\nonumber\\
    B &=(b_{ij}^k) &&\in V^*\otimes V\otimes V,\\   
    C &=(c_i^{jk}) &&\in V^*\otimes V^*\otimes V,\nonumber
    \end{alignat}
    where $\displaystyle(a_{ijk}):=a_{ijk}x^ix^jx^k$, $\displaystyle(b_{ij}^k):=b_{ij}^kx^ix^jy_k$ and $\displaystyle(c_i^{jk}):=c_i^{jk}x^iy_jy_k$.

    Any functions \(H_i\) which define a Lagrangian submanifold define an ideal with respect to the Poisson bracket---see for example \cite{WeiCoi}.  When \(H_i\) are quadratic this produces a Lie algebra \( \mathfrak{g}\) with structure constants \(g_{ij}^k\), given by the closure of the Poisson bracket from \(\mathbf{k}[W]\):
    \[ \{ H_i, H_j \} = g_{ij}^k H_k. \]
    The closure of this Lie bracket induces the following constraints on the tensors $A$, $B$ and $C$ known as an Airy structure on $V$.  

    \begin{definition}[Airy structure] \label{defairy} An Airy structure on \(V\) is a collection of tensors \eqref{airytensor} satisfying the homogeneous constraints:
    \begin{align*}
       2 \left(  b_{ji}^k - b_{ij}^k \right) &= g_{ij}^k,\\
       4\left( a_{jks} b_{it}^k -  a_{iks} b_{jt}^k \right) &=  g_{ij}^k  a_{kst}, \\
       4\left(  a_{jks} c_i^{k t} -  a_{ik s} c_j^{k t} +  b_{is}^k b_{jk}^t  -  b_{ik}^t b_{js}^k \right) & = 2 g_{ij}^k b_{k s t}^k, \\ 
       4 \left( b_{jk}^s c_i^{kt}  - b_{ik}^s c_j^{kt}\right) &= g_{ij}^k c_k^{st}.
    \end{align*}
    \end{definition}
    Airy structures were introduced by Kontsevich and Soibelman in \cite{KSoAir} and the homogeneous constraints appeared in \cite{ABCOABCD}. Their algebraic structure was generalised in \cite{BBCCNHig,BMaNew}.
 
    We study the Lagrangian \( \mathcal{L}\) in formal neighbourhoods of the origin $0\in W$.  This approach is necessary when $W$ is infinite-dimensional since in that case \(  \mathrm{Spec} \left( \mathbf{k}[W]/I(\cl) \right) \) defines a point in \(W\), for $I(\cl)=\langle H_1, H_2,...\rangle$.   In a formal neighbourhood of the origin $0\in W$, $\cl$ corresponds to a formal scheme, which we also denote by $\cl$.  It is given by completion of $\mathbf{k}[W]/I$ along a maximal ideal \(\mathfrak{m} = \langle x^{\sbt}, y_{\sbt} \rangle\) (representing zero in \(W\)). The quotient of $\mathbf{k}[W]/I$ by \( \mathfrak{m}^{k+1}\) corresponds to the \(k\)-th formal neighbourhood. So the colimit of the quotient gives  
    \[ \cl = \colim_n  \mathrm{Spec} \left( \mathbf{k}[W]/\{I,\mathfrak{m}^{n+1}\} \right) \]
    via the projective limit functor. 

    The Lagrangian $\cl$ is realised as a graph via a fixed point iteration as follows.
    Put $H_i=-y_i+\hat{H}_i$ so that $\hat{H}_i$ is quadratic in $\{x^j,y_k\}$.  The image of $\cl$ in the $n$-th formal neighbourhood of $0\in W$ is the graph:
    \begin{equation}  \label{graph}
    \{y_i^{(n)}=a_{ijk}x^jx^k+2b_{ij}^ka_{k\ell m}x^jx^\ell x^m+...\mid i=1,2,3,...\}.
    \end{equation}
    where the polynomial $y_i^{(n)}$ is obtained iteratively by
    $$y_i^{(n+1)}=\hat{H}_i(x^j,y_k^{(n)}).
    $$
    We have $y_i^{(1)}=0$ hence $y_i^{(2)}=a_{ijk}x^jx^k$ and $y_i^{(3)}$ is the cubic expression above.  This procedure produces $y_i^{(n)}$ as a degree $n$ polynomial in $x^j$ defined in the $n$th formal neighbourhood of $0\in W$, for any $n$.  Since $y_i^{(n+1)}$ and $y_i^{(n)}$ agree up to degree $n$, we can drop the superscript $y_i^{(n)}$ and write $y_i(\{x^\bullet\})$ when the $n$th formal neighbourhood is understood.
    
    \begin{example}
    \label{example:conic_classical}
    Consider the conic \(-y+x^2 + 2 xy + y^2=0\). Solving for \(y\), and taking the formal expansion of the square root gives
    \( y(x) = u_0(x) = x^2 + 2 x^3 + 5 x^4 + \cdots + \frac{(2n)!}{(n+1)!n!} x^{n+1}  + \cdots\). The coefficients are Catalan numbers, which count rooted binary trees.
    \end{example}
    
    Quite generally, any Lagrangian submanifold can be represented locally via a generating function.  The restriction of a primitive of a symplectic form to a Lagrangian submanifold is exact since integrals around contractible closed loops vanish by the
    Lagrangian condition.  Apply this to the primitive  $-y_idx^i$ of $\Omega$ in a formal neighbourhood of $0\in\cl$ to get  a function $S_0$ defined in a neighbourhood in $W$ of $0\in\cl$ satisfying
    $$y_idx^i=dS_0(\{x^i\}).
    $$
    Explicitly
    $$S_0(x)=\frac{1}{3} a_{ijk}x^ix^jx^k+\frac{1}{6} \left(b_{ij}^ka_{k\ell m}+b_{i\ell}^ka_{kj m}+b_{im}^ka_{kj\ell}\right)x^ix^jx^\ell x^m+...,
    $$
    The symmetry of $S_0(x)$ uses the closure under the Poisson bracket $\{ H_i, H_j \} = g_{ij}^k H_k$.  A consequence of the symmetry is
    $$ b_{ij}^ka_{k\ell m}+b_{i\ell}^ka_{kj m}+b_{im}^ka_{kj\ell}= b_{ji}^ka_{k\ell m}+b_{j\ell}^ka_{ki m}+b_{jm}^ka_{ki\ell}
    $$
    which agrees with the constraints in Definition~\ref{defairy}. For example we see that
    $$\frac{\partial}{\partial x^i} S_{0,4}=\frac{4}{6}(b_{ij}^ka_{k\ell m}+b_{i\ell}^ka_{kj m}+b_{im}^ka_{kj\ell})x^jx^\ell x^m= 2 b_{ij}^ka_{k\ell m}x^jx^\ell x^m=y_i^{(4)}
    $$
    as required.  

    \subsection{Tate spaces associated to a curve in a symplectic surface}
    \label{section:universal}

    In this section a bundle of Tate spaces is associated to curves in a symplectic surface.  Following \cite{KSoAir}, define the symplectic (Tate) vector space 
        $$ W_{\text{Airy}}=\left\{ J=\sum_{n\in\bz} J_nz^{-n}\frac{dz}{z}\mid J_0=0, \exists N \text{ such that }J_n=0,\ n>N\right\}
        $$
        with symplectic form
        $$\Omega_{W_\text{Airy}}(\eta_1,\eta_2)=\Res_{z=0}f_1\eta_2,\quad df_1=\eta_1,\eta_2\in W_\text{Airy}.
        $$
        The skew-symmetric bilinear form $\Omega_{W_\text{Airy}}$ is translation-invariant hence closed.  It is non-degenerate because for $J=\sum_{n\geq k} J_nz^{-n}\frac{dz}{z}\in W_\text{Airy}$, with $J_k\neq 0$, then $\Omega_{W_\text{Airy}}(J,z^k\frac{dz}{z})\neq 0$.  The locally holomorphic differentials define $L_{\text{Airy}}=\{J\mid J_n=0,n>0\}\subset W_{\text{Airy}}$, which is a Lagrangian subspace tangent to the quadratic Lagrangian $\cl_{\text{Airy}}\subset W_{\text{Airy}}$ defined in Example~\ref{quadlagKW}.
     
     Given a compact curve $\Sigma$, a non-empty finite subset $R\subset\Sigma$, and local coordinates $z_\alpha$ defined in a neighbourhood of $\alpha\in R$, define $W=(W_{\text{Airy}})^R$.  Each copy of $W_{\text{Airy}}$ uses the local coordinate $z_\alpha$.  The subspace $L=(L_{\text{Airy}})^R\subset W$ consists of locally holomorphic differentials.  Define $G_\Sigma\subset W$ by
     \begin{equation} \label{gsigma}
         G_\Sigma=\{ \eta\in H^0(\Sigma,\Omega^1(\Sigma-R)) \text{\ meromorphic on }\Sigma\mid\Res_{r\in R}\eta=0\}
     \end{equation}
    where we identify $G_\Sigma$ with its image under the injective map $G_\Sigma\to W$.  Given a choice of Torelli basis on $\Sigma$, define
    \begin{equation} \label{vsigma}
        V_\Sigma=\{ \eta\in G_\Sigma\mid\oint_{a_i}\eta=0,\ i=1,...,g\}.
    \end{equation}
    We have $L\oplus V_\Sigma=W$, proven in \eqref{Vlagcomp}, hence $V_\Sigma$ defines a polarisation
    $$W\cong V_\Sigma\oplus V_\Sigma^*.
    $$
  
  A family of pairs $(\Sigma,R)$ is obtained naturally out of a foliated symplectic surface.
    Let  $(X,\omega, \mathcal{F})$ be a symplectic surface with a Lagrangian foliation $\cF$.  Consider a curve \( \Sigma \subset X \). The curve $\Sigma\subset X$ and choice of Torelli basis determines $V_\Sigma\subset W$, defined to be those residueless differentials on $\Sigma$ with zero $a$-periods.  Recall that $\mathcal{H}\rightarrow \mathcal{B}$, defined in \eqref{cohvb}, is a vector bundle with fibre $\mathcal{H}_{[\Sigma]} = H^1(\Sigma;\mathbb{C})\cong\bc^{2g}$  and $\bg\rightarrow \mathcal{B}$ is a vector bundle with fibre $G_{\Sigma}$.  Define $[.]:\bg\rightarrow \ch$ which maps a residueless meromorphic differential to its cohomology class.
    
   The quadratic Lagrangian $\cl_{\text{KS}}\subset\widehat{W}$ defined in a formal neighbourhood of $0\in W$ by \eqref{quadlagINTRO} can be alternatively defined via the following residue constraints \cite{KSoAir}.  Choose local FD coordinates $(u_\alpha,v_\alpha)$ in a neighbourhood \(U_\alpha \subset X\) of $\alpha$ satisfying the properties of Definition~\ref{specoord}.

   A point $\eta\in\cl_{\text{KS}}$ satisfies the following residue constraints:
    \begin{alignat}{3}  
    \label{rescon1}
        &\Res_\alpha\left(v_\alpha-\frac{\eta}{du_\alpha}\right)u_\alpha^mdu_\alpha &= 0,\quad m\geq 1,\\
        &\Res_\alpha\left(v_\alpha-\frac{\eta}{du_\alpha}\right)^2u_\alpha^mdu_\alpha &= 0,\quad m\geq 0. \label{rescon2}
    \end{alignat} 
    The condition that the differential $\Res_\alpha\eta u_\alpha^mdu_\alpha=0$ for $m\geq 0$ is equivalent to $\eta$ having skew-invariant principal part under each local involution $\sigma_\alpha$ defined by $\cf$.\\

    It is convenient to express the quadratic Lagrangian $\cl_{\text{KS}}\subset W$ in the form of Section~\ref{airystr} with respect to the following choice of Darboux coordinates.
    \begin{definition}  \label{darbcoord}
    Given $(\Sigma,R)$, choose Darboux coordinates $\{x^i,y_i\}$ on $W$ satisfying $\Omega=dx^i\wedge dy_i$, $x^i\in V_\Sigma\subset W\cong W^*$ such that
    $$\frac{1}{2\pi i}\oint_{b_j}x^k=\delta_{jk},\quad j\in\{1,...,g\},\ k\in\bn
    $$
    and $y_i\in L\subset W$.
    \end{definition}
The coordinates $\{x^i,y_i\}$ satisfy the following properties.
\begin{enumerate}
    \item $[x^i]=0$ for $i>g$.
    \item $y_i=\omega_i$, the normalised holomorphic differential, for $i=1,...,g$.
\end{enumerate}
The first property uses the fact that $\oint_{a_j}x^i=0$ since $x^i\in V_\Sigma$, and combined with $\oint_{b_j}x^i=0$, we see that all periods vanish hence so does the cohomology class $[x^i]$.

The second property uses $\Omega_W=dx^i\wedge dy_i$ and the Riemann bilinear relations to deduce $\Omega_W(\omega_i,x^j)=\delta_{ij}$.  Hence $y_i=\omega_i$ follows from the nondegeneracy of $\Omega_W$.
    
For the existence of coordinates $\{x^i,y_i\}$ satisfying the conditions of Definition~\ref{darbcoord}, use the fact that the map $G_\Sigma\to H^1(\Sigma;\bc)$ is surjective. So there exists $x^i$, $i=1,...,g$ satisfying $\frac{1}{2\pi i}\oint_{b_j}x^i=\delta_{ij}$.  Then for $i>g$, complete $\{x^i\mid i=1,...,g\}$ to a basis of meromorphic differentials; for $i=i(\alpha,n)$ where $\alpha\in R$ and $n\in\bn$, let \[x_0^i=d(z_{\alpha}^{-n})+ \text{holomorphic terms} \in V_\Sigma\] 
(by the Hodge decomposition theorem) and also let  
\[ x^i=x_0^i-\sum_{j=1}^gx^j\frac{1}{2\pi i}\oint_{b_j}x_0^i. \]  
In terms of $x^i$, we have \[y_i^0=\frac{1}{n}d(z_\alpha^n)+\sum_{j=1}^g\omega_j\frac{1}{2\pi i}
\oint_{b_j}x_0^i\] for $i>g$.  Then define \[y_i=y^0_i-\sum_{j=1}^gy_j\Omega_W(y^0_i,x^j).\]
The coordinates $\{x^i,y_i\}$ from Definition~\ref{darbcoord} are not unique, since, for example, $x^1\mapsto x^{1}+x^{g+1}$ and $x^i\mapsto x^i$, $i>1$ (which induces a linear change of the variables $y_i$) also satisfies the conditions.  However, the set of vector fields
$$\frac{\partial}{\partial x^1},\frac{\partial}{\partial x^2},...,\frac{\partial}{\partial x^g},
$$
is well-defined independent of the ambiguity in the choice of coordinates $\{x^i,y_i\}$.  This can be seen in two ways.  We have
$$\frac{\partial}{\partial x^i}=\omega_i,\quad i=1,...,g
$$
where the normalised holomorphic differential $\omega_i$ represents a vector field independent of coordinate choices.  Or more directly,
the linear change $x^1\mapsto x^1+x^{g+1}$ and $x^i\mapsto x^i$ induces $\frac{\partial}{\partial x^1}\mapsto \frac{\partial}{\partial x^1}$ and $\frac{\partial}{\partial x^{g+1}}\mapsto \frac{\partial}{\partial x^{g+1}}+\frac{\partial}{\partial x^1}$ leaving $\frac{\partial}{\partial x^1}$ invariant.

    \subsection{Symplectic reduction}  \label{sympred}

        Consider a symplectic manifold $(M,\omega)$ that admits a proper Hamiltonian action of an abelian Lie group $G$ and an invariant moment map $\mu:M\to\mathfrak{g}^*$.  The moment map is characterised by
        \begin{equation}  \label{mom}
            \omega(\xi_u(m),\cdot)=d\langle\mu(m),u\rangle
        \end{equation}
        where $u\in\mathfrak{g}$ defines the vector field $\xi_u$ on $M$ by
        $\xi_u(m)=\frac{d}{dt}(g(t)\cdot m)|_{t=0}
        $ and $g'(0)=u$.
        
        For any regular value $a$ of $\mu$, define the symplectic quotient
        $$ M\hspace{-1mm}\sslash\hspace{-1mm} G:=\mu^{-1}(a)/G.
        $$
        Then $M\hspace{-1mm}\sslash\hspace{-1mm} G$ inherits a symplectic form (depending on $a$).
        
        Apply these ideas to a symplectic vector space $(W,\omega)$ equipped with a translation-invariant symplectic form $\omega$.  Let $U\subset W$ be an isotropic subspace, so $\omega|_{U}=0$ or equivalently $U\subset U^\perp$.  Then $U$ acts on $W$ by translations, and hence preserves $\omega$.  The moment map $\mu$ is given by the quotient map
        $$0\to U^\perp\to W\stackrel{\mu}{\to} U^*\to 0
        $$
        because $\omega(u,v)=\langle\mu(v),u\rangle,\ \forall  u,v\in W$ agrees with \eqref{mom} (since $u=\xi_u(m)$ and $d\langle\mu(m),u\rangle(v)=\langle\mu(v),u\rangle$).  Hence $U^\perp=\mu^{-1}(0)$ and the symplectic quotient is given by:
        $$ W\hspace{-1mm}\sslash\hspace{-1mm} U:=U^\perp/U.
        $$
        The symplectic form $\overline{\omega}$ on $W\hspace{-1mm}\sslash\hspace{-1mm} U$ is defined by $\overline{\omega}(\bar{v}_1,\bar{v}_2)=\omega(v_1,v_2)$ where $v_i\in U^\perp$ is any lift of $\bar{v}_i\in U^\perp/U$.  The right hand side is independent of the lift since $\omega(v_1+u,v_2)=\omega(v_1,v_2)$ for any $u\in U$.  The 2-form $\overline{\omega}$ is closed since it is translation invariant.  It is non-degenerate since $\overline{\omega}(\bar{v}_1,\bar{v}_2)=\omega(v_1,v_2)=0$ for all $\bar{v}_2$, hence all $v_2$ implies that $v_1\in U$ hence $\bar{v}_1\equiv 0$.
        
        To apply this to $G_\Sigma\subset W$ defined by the pair $(\Sigma,R)$ we need the following.
        \begin{lemma}[\cite{KSoAir}]  \label{coisotropic}
        $G_\Sigma\subset W$ is coisotropic.
        \end{lemma}
        \begin{proof}
        We first show that $G_\Sigma$ and $L$ intersect transversally, i.e.
        \begin{equation}  \label{transversal}
        G_\Sigma+L=W.
        \end{equation} 
        Define 
        \[W_k=\{ J\in W\mid J_n=0,n\geq k\}\]
        so $H=W_1\subset W_2\subset...\subset W_k\subset W_{k+1}\subset ... \subset W=\cup_{k>0} W_k$.  We have $\dim(W_k/W_{k-1})=|R|$ and by Riemann-Roch this vector space can be represented by elements of $G_\Sigma$ since
        \begin{align*}
        \dim(W_k\cap G_\Sigma)-\dim(&W_{k-1}\cap G_\Sigma)\\
        &=\dim H^0(\Sigma,K_\Sigma(kR))-\dim H^0(\Sigma,K_\Sigma((k-1)R))\\
        &=k|R|+1-g-((k-1)|R|+1-g)=|R|.
        \end{align*}
        Hence $\left(W_k\cap G_\Sigma\right)+L=W_k$ and 
        \eqref{transversal} follows by taking the union over $k>0$.
        
        Define $V_\Sigma \subset G_\Sigma$ to consist of those differentials with vanishing $a$-periods.  It is easy to see that the proof of \eqref{transversal} can be adjusted to yield:
        \begin{equation} \label{Vlagcomp}
         V_\Sigma \oplus L= W,
        \end{equation}
        since the elements of $W_k\cap G_\Sigma$ can be chosen to be normalised to have vanishing $a$-periods.
        
        Let $\eta_1+\ell\in G_\Sigma^\perp$ for $\eta_1\in V$ and $\ell\in L$.  Then $\Omega_W(\eta_1+\ell,\eta_2)=0$ for all $\eta_2\in V$ since $ \eta_1+\ell$ annihilates all elements of $G_\Sigma$, in particular those from $V$.  But $\Omega_W(\eta_1,\eta_2)=0$ by the Riemann bilinear relations:
        \begin{equation}  \label{RiemBilRel}
        \sum\Res_{r\in R} f_1\eta_2=\frac{1}{2\pi i}\sum_{j=1}^g\oint_{b_j}\eta_1\oint_{a_j}\eta_2-\oint_{b_j}\eta_2\oint_{a_j}\eta_1=\frac{1}{2\pi i}\int_\Sigma[\eta_1]\wedge[\eta_2]
        \end{equation}
        since all $a$-periods vanish in the middle expression of \eqref{RiemBilRel}.  Here $df_1=\eta_1$ for a locally defined function $f_1$.  Hence $\Omega_W(\ell,\eta_2)=0$ for all $\eta_2\in V$.  But $\omega$ is symplectic, so for any non-zero $\ell\in L$ there is $\eta_2\in V$ such that $\Omega_W(\ell,\eta_2)\neq 0$.  We conclude that $\ell=0$ so $\eta_1+\ell=\eta_1\in V$, hence 
        \[G_\Sigma^\perp\subset V\subset G_\Sigma\]
        as required.
        \end{proof}
        Strengthening 
        Lemma~\ref{coisotropic}, elements of  $G_\Sigma^{\perp}$, are exact, i.e.
        $$
        G_\Sigma^{\perp}=\{ \eta=df,\  f \text{\ holomorphic on }\Sigma-R\}
        .$$
        To show this, first note the inclusion of exact differentials into $G_\Sigma^{\perp}$ follows
        from the fact that if $\eta_1=df_1$ for a {\em global} meromorphic function $f_1$, then $\sum\Res_{r\in R} f_1\eta_2=0$ since it is the sum of the residues of the meromorphic differential $f_1\eta_2$. 
        
        For the other direction, by Lemma~\ref{coisotropic}, any $\eta_1\in G_\Sigma^{\perp}$ lives in $G_\Sigma$ hence it is (the local expansion of) a globally defined meromorphic differential on $\Sigma$.  The $b$-periods of $\eta_1$ can be calculated using \eqref{RiemBilRel}.  Let $\omega_i$, $i=1,...,g$ be the normalised holomorphic differentials on $\Sigma$, so $\oint_{a_j}\omega_i=\delta_{ij}$.  We have
        $$
        \oint_{b_j}\eta_1=\oint_{a_j}\omega_j\oint_{b_j}\eta_1-\oint_{a_j}\eta_1\oint_{b_j}\omega_j=\Omega_W(\eta_1,\omega_j)=0
        $$
        since $\omega_j\in G_\Sigma$ and $\eta_1\in G_\Sigma^{\perp}$.  But the residues of $\eta_1$ and all of its $a$-periods and $b$-periods vanish.  Hence 
        \[f(p):=\int^p_{p_0}\eta_1\]
        is well-defined and $\eta_1$ is exact.
        
        Thus the symplectic quotient of $W$ is given by
        $$ W\hspace{-1mm}\sslash\hspace{-1mm} G_\Sigma^\perp:=G_\Sigma/G_\Sigma^\perp\cong\ch_\Sigma= H^1(\Sigma;\bc),
        $$
        where the isomorphism uses the fact that elements of $G_\Sigma$ define cohomology classes on $\Sigma$, the quotient by exact differentials sends a meromorphic differential (with zero residues) to its cohomology class, and the map is surjective.  Clearly
        \[V_\Sigma\to\overline{V}_\Sigma=V_\Sigma/G_\Sigma^\perp\]
        where $\overline{V}_\Sigma\subset\ch_\Sigma$ consists of those cohomology classes with vanishing $a$-periods.
        
        
\subsubsection{The quadratic Lagrangian \texorpdfstring{$\cl_{\text{KS}}$}{L}}        
    Given a symplectic quotient
    \[ M\hspace{-1mm}\sslash\hspace{-1mm} G:=\mu^{-1}(a)/G\]
    and a Lagrangian submanifold $\cl\subset M$, if $\cl$ intersects $\mu^{-1}(a)$ transversally, then the quotient of $\cl\cap\mu^{-1}(a)$ defines a Lagrangian submanifold of $\mu^{-1}(a)/G$.  
    
    The intersection $\cl_{\text{KS}}\cap G_\Sigma$ is transversal since $W=G_\Sigma+L=G_\Sigma+T_0\cl_{\text{KS}}$ which is proven in  \eqref{transversal}.   
    To make sense of the quotient of $\cl_{\text{KS}}\cap G_\Sigma$ by $G_\Sigma^\perp$, we need to treat the symplectic reduction of $(W,\Omega_W)$ algebraically, since $\cl_{\text{KS}}$ lives in a formal neighbourhood of $0\in W$.
    The quotient
    $$G_\Sigma^\perp\to G_\Sigma\to \ch_\Sigma
    $$
    corresponds to the ring homomorphism
    \begin{align*}
    \mathbf{k}[\ch_\Sigma]&\to \mathbf{k}[G_\Sigma]^{G_\Sigma^\perp}\\
    (z^i,\w_i)&\mapsto(\oint_{a_i},\oint_{b_i})
    \end{align*}
    where $\mathbf{k}=\bc$ and $\mathbf{k}[V]=\bigoplus_k \mathrm{S}^k(V^*)$ is the ring of regular functions on the vector space $V$.
    
    In terms of the coordinates defined in Definition~\ref{darbcoord}, we have:
    $$
    \mathbf{k}[G]=\mathbf{k}[W]/\{x^i\mid i>g\}
    $$
    since the restriction $x^i|_{G_\Sigma}$, for $i\in\bn$ depends only on its cohomology class $[x^i]\in H^1(\Sigma;\bc)$.  In particular, the ambiguity in the choice of coordinates $x^1,...,x^g$ in Definition~\ref{darbcoord} disappears under restriction to $G_\Sigma$.  Also
    $$
    \mathbf{k}\left[\cl_{\text{KS}}\right]=\mathbf{k}[W]/\{y_i=a_{ijk}x^jx^k+2b_{ij}^ka_{k\ell m}x^jx^\ell x^m+...\},
    $$
    which is defined in a formal neighbourhood of $0\in W$, where the series for $y_i$ are defined in \eqref{graph}.
    From the commutative square
    $$\begin{array}{ccc}
         \mathbf{k}[W]&\longrightarrow& \mathbf{k}[G_\Sigma]\\
         \downarrow&&\downarrow\\
    \mathbf{k}\left[\cl_{\text{KS}}\right]&\longrightarrow&\mathbf{k}\left[{\cl_{\text{KS}}\cap G_\Sigma}\right] 
    \end{array}
    $$
    we find that
    $$
    \mathbf{k}\left[\cl_{\text{KS}}\cap G_\Sigma\right]=\mathbf{k}[W]/\{y_i=a_{ijk}x^jx^k+...,x^m=0,m>g\}\cong \mathbf{k} \llbracket x^1,...,x^g \rrbracket.
    $$
    
    Compose $\mathbf{k}[\ch_\Sigma]\to \mathbf{k}[G_\Sigma]$ with the right vertical arrow in the commutative square to get:
    \begin{equation}  \label{BLcorr}
       \begin{array}[t]{ccc}
            \mathbf{k}[\ch_\Sigma]&\to& \mathbf{k}\left[\cl_{\text{KS}}\cap G_\Sigma\right]  \\
            (z^i,\w_i)&\mapsto&(x^i,y_i+\tau_{ij}x^j)
       \end{array}
    \end{equation}
    which is a map from a formal neighbourhood of $0\in\ch_\Sigma$ to a formal neighbourhood of $0\in W$.
    In Section~\ref{formalconv} it is proven that the kernel of \eqref{BLcorr} is given by the ideal $\{\w_i=\w_i(z^1,...,z^g)\}$ hence \eqref{BLcorr} defines an isomorphism $\mathbf{k}[\widehat{\cb}]\cong \mathbf{k}\left[\cl_{\text{KS}}\cap G_\Sigma\right]$.
    
    \subsubsection{Choice of quadratic Lagrangian}  \label{BGW}
An Airy structure is equivalent to the choice of a quadratic Lagrangian.  The work of Kontsevich and Soibelman \cite{KSoAir} is based on the quadratic Lagrangian $\cl_{\text{KS}}=\cl_{\text{Airy}}^R$ where $\cl_{\text{Airy}}$ is built from the Kontsevich-Witten tau function.  In place of $\cl_{\text{Airy}}$ we can use a quadratic Lagrangian $\cl_{\text{Bessel}}$ built from the Br\'{e}zin-Gross-Witten tau function of the KdV hierarchy which arises out of a unitary matrix model studied in \cite{BGrExt,GWiPos}. For $m=0,1,...$ the operators 
$$L_m=-\frac12\frac{\partial}{\partial x^{2m+1}} +\frac{\hbar}{4} \hspace{-2mm}\mathop{\sum_{i+j=2m}}_{i,j\text{ odd}} \hspace{-2mm} \frac{\partial^2}{\partial x^i \partial x^j}
 +\frac12\mathop{\sum_{i=1}}_{i\text{ odd}}^\infty i x^i \frac{\partial}{\partial x^{i+2m}}+ \frac{1}{16} \delta_{m,0},
$$
satisfy Virasoro relations
\[
[L_m, L_n] = (m-n) L_{m+n}, \quad \text{for } m, n \geq 0,
\]
The Br\'{e}zin-Gross-Witten tau function is uniquely defined by
$$
L_m Z^{\text{BGW}}(\hbar,x^1,x^3,...)=0,\quad m=0,1,2,...
$$ 
and the initial condition
$$\log Z^{\text{BGW}}(x^1,0,0,...)=\frac18\log(1-x^1).$$
Analogous to $Z^{\text{KW}}$, the tau function $Z^{\text{BGW}}$ is shown in \cite{NorNew} also to be a generating function for intersection numbers over $\overline{\cm}_{h,n}$.  The Virasoro operators give rise to the quadratic Lagrangian $\cl_{\text{Bessel}}\subset W_{\text{Airy}}$ defined by the ideal:
\begin{align*}
   H_k(x^{\sbt},y_{\sbt})&=-y_k,\quad k\in\bz^+_{\text{even}},\\
H_k(x^{\sbt},y_{\sbt})& = \hbar L_{\frac{k-1}{2}}\left(x^{\sbt},\hbar\frac{\partial}{\partial x^{\sbt}}\right)|_{\hbar\frac{\partial}{\partial x^i}=y_i}\quad k\in\bz^+_{\text{odd}},\\
&=-\tfrac12 y_k +\tfrac14 \hspace{-2mm}\mathop{\sum_{i+j=k-1}}_{i,j\text{ odd}} \hspace{-2mm} y_i y_j
 +\tfrac12\mathop{\sum_{i=1}}_{i\text{ odd}}^\infty i x^i y_{i+k-1}+ \tfrac{1}{16} \delta_{k,3}.
\end{align*} 
Define $\cl=\cl_{\text{Bessel}}^R$. More generally, one can also combine a product of a combination of copies of $\cl_{\text{Airy}}$ and $\cl_{\text{Bessel}}$.  This produces topological 
recursion on irregular spectral curves \cite{DNoTop} with local behaviour at points in $R$ giving topological 
recursion over the Bessel curve \cite{DNoTopB}.  In the case $\cl=\cl_{\text{Bessel}}^R$ the tensor $A_\Sigma=0$ and the Airy structure consists of the tensors $B_\Sigma$ and $C_\Sigma$.  There are residue constraints analogous to \eqref{rescon1} which define $\cl$:
\[\Res_\alpha\left(dv_\alpha-\eta\right)u_\alpha^m = 0=\Res_\alpha\left(dv_\alpha-\eta\right)^2\frac{u_\alpha^m}{du_\alpha} = 0,\quad m\geq 1. 
\]
If $\cl=\cl_{\text{Airy}}^{R_1}\times_{\mathrm{Spec} \,\mathbf{k}} \cl_{\text{Bessel}}^{R_2}$ where $R_1\cup R_2=R$ then the residue constraints above, respectively the residue constraints \eqref{rescon1}, are used at $\alpha\in R_2$, respectively $\alpha\in R_1$.

    \subsection{Quantum Airy structures}  \label{qairy}
    
    Now let \( \mathcal{D} = \mathbf{k}\llbracket x^{\sbt}, \hbar \partial_{\sbt} \rrbracket \llbracket \hbar \rrbracket \) be a graded Weyl algebra with Lie bracket given by \( [ x^i , x^j] =  [\hbar \partial_i, \hbar \partial_j ] = 0\) and \( [ \hbar \partial_i , x^j ] =  \hbar \delta^j_i \).

    Consider the differential operators \(\widehat{H}'_i \in \mathcal{W} \) 
    \begin{align*}
            \widehat{H}'_i = \hbar \partial_i + a_{ijk} x^j x^k +   2 \hbar \,  b_{ij}^k  \, x^j \partial_k + \hbar^2 c_{i}^{jk} \partial_j \partial_k.
    \end{align*}
        

    Kontsevich and Soibleman define a \emph{quantum Airy structure} as the deformation quantisation of the classical Airy structure on the Lagrangian \( \mathcal{L}\). Deformation quantisation is a functor which replaces the commutative algebra \( \mathbf{k}\llbracket W \rrbracket \) with the non commutative Weyl algebra of differential operators \( \mathcal{W} \). First the coordinates are mapped by \( x^i \rightarrow x^i \) and \( y_i \rightarrow \hbar \partial_i \) where \(\partial_i \) can be identified with the vector field, or derivation, \( \frac{\partial}{\partial x^i}\) when \( \hbar \) is invertible.  
    Further, the Poisson Lie algebra \( \mathfrak{g}\) with Poisson bracket has to be identified with the Lie algebra structure of the \(\widehat{H}_i\) with a Lie bracket. A necessary condition to do this, is that the second cohomology vanishes, \(H^2(\mathfrak{g}, \mathbf{k})=0\) \cite{KSoAir}. This is a choice of central extension of \(\mathfrak{g}\) and the Lie algebra structure of the quantum and classical cases coincide. So  
    \[ H_i \rightarrow \widehat{H}_i = \widehat{H}'_i + \hbar \varepsilon_i\] 
    and 
    \[ [\widehat{H}_i, \widehat{H}_j] = \hbar \{ H_i , H_j \}_{x^{\sbt}\rightarrow x^{\sbt}, y_{\sbt} \rightarrow \hbar \partial_{\sbt}} + \hbar g_{ij}^k \varepsilon_k.\]
    
    \begin{definition}
    A \emph{quantum Airy structure} is the collection of \( \widehat{H}_i := \widehat{H}'_i + \hbar \varepsilon_i\), and an extra constraint:
    \[ 2 \left( a_{jst} \, c_i^{st} - a_{ist} \, c_j^{st} \right) = g_{ij}^k \varepsilon_k.\]
    \end{definition}
    
    When a quantum Airy structure arises from deformation quantisation, this gives rise to a wavefunction supported on \(\mathcal{L}\). A wavefunction is a generator of a cyclic module \(\mathcal{E}\) over \(\mathcal{D}\), given by the quotient \( \mathcal{E} = \mathcal{D} / \mathcal{D} \langle \widehat{H}_i \rangle  \). This module encodes the solution to the \(\widehat{H}_i\) acting as operators on \( \mathbf{k}\llbracket x^{\sbt} \rrbracket \llbracket \hbar \rrbracket \). The wavefunction \( \psi_{\mathcal{L}} \in \mathbf{k}\llbracket x^{\sbt} \rrbracket \llbracket \hbar \rrbracket\) is computed using the ansatz 
    \[ \psi_{\mathcal{L}} = \exp(S(x^{\sbt}))\]
    where
    \[ S(x^{\sbt}) = \sum_{h \geq 0} \hbar^{h-1} S_h(x^{\sbt}), \quad S_h(x^{\sbt}) \in \mathbf{k} \llbracket x^{\sbt}\rrbracket\]
    and solving the differential equations
    \[ \widehat{H}_i \exp(S(x^{\sbt})) = 0. \] 
    Modulo \(\hbar \), \( \psi_{\mathcal{L}}\) is a function on \( \widehat{\mathcal{L}}\).

    \begin{example} Consider the quantised conic with \( \varepsilon = 0\):
    \[ \left( - \hbar \frac{\partial}{\partial x} + x^2 + 2 \hbar x \frac{ \partial }{\partial x}+ \hbar^2 \frac{\partial^2}{\partial^2 x} \right) \psi_{\mathcal{L}}(x) = 0. \]
    Computing some terms
    \[ \psi_{\mathcal{L}}(x) = \exp \left( \frac{1}{\hbar} \int dx \, ( u_0(x)  + \hbar \, u_1(x) + \mathcal{O}(\hbar^2) ) \right), \]
    with \(  u_0(x)\) as example \ref{example:conic_classical}, and \(u_1(x) = 2 x + 10 x^2 + \cdots + (4^n - \frac{(2n)!}{(n!)^2}) x^n + \cdots \) a generating function for counting numbers of rooted two-face \(n\)-edge maps in the plane, (1-loop Feynmann diagrams) \cite{feynmaps}, and in general rooted \(u_h\) counts \(h\)-face \(n\)-edge maps. Then 
    \[ S_h(x) = \int dx\, u_h(x) . \]
    \end{example}

    Kontsevich and Soibelman \cite{KSoAir} prove that the coefficients of \(S_{h,n}\) of \(S\) satisfy \emph{abstract topological recursion} defined as follows. Denote \( S_{h,n; \sbt:i} = \partial_i S_{h,n; \sbt }\). Apply the \( \widehat{H}_i\) to \(\psi_{\mathcal{L}} \)  and solve for \(0 \): 
    \begin{align*}
        & a_{ijk} x^j x^k + \sum_h \bigg( 2 \, \hbar \, b_{ik}^{j} \sum_n  S_{h,n;j} x^k + \\
        &\hbar^2 c_i^{jk}  \left(\sum_n S_{h,n;j,k} + \sum_n S_{h,n;j} S_{h,n;k} \right) - \hbar \, S_{h,n;i} + \hbar \, \epsilon_i \bigg)  = 0.
    \end{align*}
    Gathering coefficients:
    \begin{align}  \label{abstractTR}
        S_{h,n;i,i_1,\dots,i_{n-1}} =&  \, c^{jk}_i S_{h-1,n+1;j,k,i_1, \dots i_{n-1}} \\
        \nonumber +& \, 2 \sum_{\alpha = 1 }^{n-1} b^k_{i \, i_\alpha} S_{h,n-1;k i_{\{1,\dots n-1\}}/\{\alpha\}}  \\ 
        \nonumber +& \sum_{\substack{h_1 + h_2 = h  \\ I_1 \sqcup I_2 = \{1, \dots, n-1\}} } c^{jk}_i S_{h_1, |I_1| + 1; j}\, S_{h_2, |I_2|+1;k} 
    \end{align}
    produces a recursive formula known as abstract topological recursion. The sum of \(h-1\) and \(h_1+h_2 = h\) terms gives a resemblance to topological recursion. 
    
    The symmetry of $S_{h,n}$ for $h=0$ uses the same argument as for the classical case, which uses closure of the Poisson bracket $\{ H_i, H_j \} = g_{ij}^k H_k.$  For higher genus, the argument is given in \cite[Theorem 2.4.2]{KSoAir} for finite dimensional $V$ which suffices here since $V_\Sigma$ is the union of finite dimensional subspaces graded by the degree of poles, and $S_{h,n}$, and all $S_{h',n'}$ for $2h'-2+n<2h-2+n$ live inside one of these finite dimensional subspaces.  
    
    Topological recursion of Eynard and Orantin \cite{eynard_orantin}, can be seen as a particular specialisation of abstract topological recursion. Restricting to the odd \(H_i\) recovers topological recursion:
    \[ S^{\text{odd}}_{h,n} \rightarrow \frac{1}{n!}\omega_{h,n}.\]

    \begin{remark}
    The constructions of $\omega_{h,n}$ via the Eynard-Orantin recursion \eqref{rec} and via abstract topological recursion \eqref{abstractTR} produce different proofs of the symmetry of $\omega_{h,n}$ which demonstrates a departure between the two constructions.  The proof in \cite{eynard_orantin} using \eqref{rec} expresses the difference $\omega_{h,n}(p_1,p_2,...,p_n)-\omega_{h,n}(p_2,p_1,...,p_n)$ as a sum of a collection of terms which are shown to vanish rather non-trivially.  The rather elegant proof in \cite{KSoAir} is a consequence of the fact that the Hamiltonians that define a Lagrangian submanifold generate an ideal, expressed above via  $\{ H_i, H_j \} = g_{ij}^k H_k$. is rather elegant. 
    \end{remark}

    
    \section{Formal and convergent series}   \label{formalconv}

Let $(X,\Omega,\cf)$ be a foliated symplectic surface, $\Sigma\subset X$ and $\cb$ the deformation space of $\Sigma$ in $X$.   Recall from \eqref{vsigma} that $\Sigma$ defines a polarisation $V_\Sigma\subset W$ of the symplectic vector space $W$ of locally defined residueless meromorphic differentials on $\Sigma$.

In this section we study the following commutative diagram from \cite{KSoAir}:
\begin{equation}  \label{lagmap}
\begin{array}{ccccl}\cl_{\text{KS}}\cap G_{\Sigma}&\longrightarrow&G_{\Sigma}&\longrightarrow& V_\Sigma\oplus V_\Sigma^*\cong W\\
\downarrow&& \downarrow&&\\
\widehat{\cb}_{[\Sigma]}&\stackrel{}{\longrightarrow} &\ch_\Sigma&\stackrel{\cong}{\longrightarrow}&\overline{V}_\Sigma\oplus \overline{V}_\Sigma^*
\end{array}
\end{equation}
which shows how the quadratic Lagrangian 
\[\cl_{\text{KS}}\to W\cong V_\Sigma\oplus V_\Sigma^*\]
defined in \eqref{quadlagINTRO}, and via residue constraints in \eqref{rescon1} and \eqref{rescon2}, behaves under symplectic reduction 
\[W\hspace{-1mm}\sslash\hspace{-1mm} G_\Sigma^\perp:=G_\Sigma/G_\Sigma^\perp\cong \ch_\Sigma\]
where $\ch_\Sigma=H^1(\Sigma;\bc)$.
Its image is a formal neighbourhood of a point of the Lagrangian embedding of a neighbourhood $U_{[\Sigma]}\subset\cb$ of $[\Sigma]\in\cb$
\[U_{[\Sigma]}\to\ch_\Sigma\cong\overline{V}_\Sigma\oplus \overline{V}_\Sigma^*\]defined in \eqref{Bemb}.  More precisely, $\cl_{\text{KS}}$ is defined in a formal neighbourhood of $0\in W$ and its intersection with the zero level set of the moment map maps to a formal neighbourhood $\widehat{\cb}_{[\Sigma]}\stackrel{\iota}{\to}\cb$ of $[\Sigma]\in\cb$.

An explicit section $\theta\in\Gamma(\widehat{\cb}_{[\Sigma]},G_\Sigma)$ of the bundle $\bg\rightarrow \mathcal{B}$ with fibre $G_{\Sigma}$ is constructed in Theorem~\ref{phisec} below.  It is given by a formal series which takes its values in $\cl_{\text{KS}}\cap G_{\Sigma}$ and in fact defines an isomorphism
\[\widehat{\cb}_{[\Sigma]}\cong\cl_{\text{KS}}\cap G_{\Sigma}\]

with inverse producing the left vertical arrow in \eqref{lagmap}.  It is proven in Theorem~\ref{phisec} that under the quotient map $\bg\to\ch$ to the bundle $\ch\rightarrow \mathcal{B}$ with fibre $\ch_{\Sigma} \cong G_{\Sigma}/G_{\Sigma}^\perp$, the section $\theta$ maps to $[\theta]=\iota^*[\theta]\in\Gamma(\widehat{\cb}_{[\Sigma]},\ch)$ which is the composition 
\begin{equation}  \label{restheta}
    \widehat{\cb}_{[\Sigma]}\stackrel{\iota}{\longrightarrow} U_{[\Sigma]}\stackrel{[\theta]}{\longrightarrow} \ch|_{U_{[\Sigma]}}
\end{equation}
and by abuse of notation it is given the same name as the analytic section $[\theta]$ defined in \eqref{cohtheta}.    The composition \eqref{restheta} defines the lower left horizontal arrow in \eqref{lagmap}.  Hence the formal series $\theta$ maps under the quotient to an analytic series $[\theta]$.   

    For $L=T_0\cl_{\text{KS}}\subset W$ given by the Lagrangian subspace of locally holomorphic differentials, defined in a neighbourhood of $R\subset\Sigma$, the symplectic form $\Omega_W$ defines a natural isomorphism of Lagrangian subspaces $L\cong V_\Sigma^*$.  Similarly, in the symplectic quotient $\ch_\Sigma$, the symplectic form $\Omega$ defines a natural isomorphism of Lagrangian subspaces $H^0(\Sigma,K_\Sigma)\cong \overline{V}_\Sigma^*$.  Define the linear map 
    \[h:H^0(\Sigma,K_\Sigma)\to L\]
    which maps a holomorphic differential to its local expansion at $R\subset\Sigma$.  For any $T\in\text{Hom}(L\otimes L,V_\Sigma)$ the square
    \[
     \begin{tikzcd} L\otimes L\arrow{r}{T} & V_\Sigma \arrow{d} \\ H^0(\Sigma,K_\Sigma)\otimes H^0(\Sigma,K_\Sigma)\arrow{r} \arrow{u}{h\otimes h} & \overline{V}_\Sigma  \end{tikzcd} 
    \]
    induces a map $T\mapsto[T\circ (h\otimes h)]\in\text{Hom}(H^0(\Sigma,K_\Sigma)\otimes H^0(\Sigma,K_\Sigma),\overline{V}_\Sigma),$
    where $[\cdot]$ is the map $V_\Sigma\to\overline{V}_\Sigma$. This defines the right vertical arrow in the following commutative diagram:
\begin{equation}  \label{tenscovder}
    \begin{tikzcd}
    V_\Sigma\otimes V_\Sigma\otimes V_\Sigma\arrow{d}\arrow{r}{\Omega_W}  &\text{Hom}(L\otimes L,V_\Sigma)\arrow{d}  \\
    \overline{V}_\Sigma\otimes \overline{V}_\Sigma\otimes \overline{V}_\Sigma  \arrow{r}{\Omega} &
    \text{Hom}(H^0(\Sigma,K_\Sigma)\otimes H^0(\Sigma,K_\Sigma),\overline{V}_\Sigma). \end{tikzcd}
\end{equation}
    
 The section $\theta$ allows us to associate vector fields over $\cb$ to vector fields over $\cl_{\text{KS}}$.  In particular, this leads to a relationship between the tensor $A_\Sigma\in V_\Sigma\otimes V_\Sigma\otimes V_\Sigma$, which is part of the Airy structure arising from $\Sigma\subset(X,\Omega,\cf)$, and the tensor $\bar{A}_\Sigma\in\overline{V}_\Sigma\otimes \overline{V}_\Sigma\otimes \overline{V}_\Sigma$ representing the Donagi-Markman cubic.  In general, for a Lagrangian submanifold of a polarised symplectic vector space $\cl\subset V\oplus V^*$, the tensor $A$ is defined via the map
    \[T_p\cl\otimes T_p\cl\to V\]
given by variation of a vector field with respect to another vector field.  This uses a canonical extension of any given vector in $T_p\cl$ to a local vector field so that covariant differentiation gives a tensor, meaning it depends only on vectors.  By relating vectors, their canonical extensions to vector fields, and covariant differentiation upstairs and downstairs in \eqref{lagmap} we prove $A_\Sigma\to\bar{A}_\Sigma$ via the map \eqref{tenscovder}.       

\subsection{The section $\theta\in\Gamma(\widehat{\cb}_{[\Sigma]},G_\Sigma)$.}
 We define a section $\theta\in\Gamma(\widehat{\cb}_{[\Sigma]},G_\Sigma)$ in terms of holomorphic differentials $\omega_i$ normalised over the $a$-periods and the topological recursion correlators $\omega_{0,n}$ calculated via \eqref{rec}.  Represent elements of the formal neighbourhood $\widehat{\cb}_{[\Sigma]}$ with respect to the coordinates $\{z^1,...,z^g\}$ defined in \eqref{zcoords} satisfying $z^i([\Sigma])=0$, and sum over indices in $\{1,...,g\}$.   For any residueless meromorphic differential $\eta$ defined on $\Sigma$, we use the normalised periods defined in \eqref{bnorm}:
\[
     \oint_{\hat{b}_k}\eta:=\frac{-1}{2\pi i}\oint_{b_k}\eta.
\]
\begin{thm}   \label{phisec}
Define a section $\theta\in\Gamma(\widehat{\cb}_{[\Sigma]},G_\Sigma)$ by
\begin{equation}  \label{thetaTR}   
\theta=z^i\omega_i-\frac12z^iz^j\oint_{\hat{b}_i}\oint_{\hat{b}_j}\omega_{0,3}-\frac{1}{3!}z^iz^jz^k\oint_{\hat{b}_i}\oint_{\hat{b}_j}\oint_{\hat{b}_k}\omega_{0,4}-...
\end{equation}
Then $\theta$ satisfies the following properties:
\begin{enumerate}
\item It takes its values in $\cl_{\text{KS}}$.
\item Its cohomology class $[\theta]\in \Gamma(\widehat{\cb}_{[\Sigma]},\ch)$ is analytic in $z^1,...,z^g$ and coincides with the local section defined in \eqref{cohtheta}.
\end{enumerate}
\end{thm}  
More precisely, $[\theta]$ is the restriction of an analytic section to a formal neighbourhood of $[\Sigma]$ given in \eqref{restheta}.  The analyticity of $[\theta]$ contrasts with the formal series for $\theta$. The proof of Theorem~\ref{phisec} is given by Propositions~\ref{phisec1} and \ref{phisec2}.

A cohomology class is characterised by its periods along a Torelli basis:
\[[\theta]=(\oint_{a_i}\theta,\oint_{b_i}\theta\mid i=1,...,g)\in\bc^{2g} \llbracket z^1,...,z^g \rrbracket .\]
The periods are:
\[\oint_{a_i}\theta=z^i,\quad\oint_{b_i}\theta=\w_i(z^1,...,z^g).
\]
\begin{corollary}  \label{bper}
The Taylor expansion of \normalfont{$\w_i(z^1,...,z^g)$} around $\{z^i=0\}$ is:
\normalfont{
\begin{equation}   \label{bperi}   
\w_i=z^j\tau_{ij}-\frac12z^jz^k\oint_{b_i}\oint_{\hat{b}_j}\oint_{\hat{b}_k}\omega_{0,3}-\frac{1}{3!}z^jz^kz^\ell\oint_{b_i}\oint_{\hat{b}_j}\oint_{\hat{b}_k}\oint_{\hat{b}_\ell}\omega_{0,4}-...
\end{equation}}
\end{corollary}
where the normalised periods \eqref{bnorm} are used in \eqref{bperi} except for the first period. 
Hence Corollary~\ref{bper} shows that 
\[\frac{\partial}{\partial z^i}\tau_{jk}=-2\pi i\oint_{\hat{b}_i}\oint_{\hat{b}_j}\oint_{\hat{b}_k}\omega_{0,3}\]
and more generally 
\[\frac{\partial^{n-2}}{\partial z^{i_1}..\partial z^{i_{n-2}}}\tau_{i_{n-1}i_n}=-2\pi i\oint_{\hat{b}_{i_1}}\oint_{\hat{b}_{i_2}}...\oint_{\hat{b}_{i_n}}\omega_{0,|I|}\] 
which proves Corollary~\ref{om3cor} and generalises the result in \cite{BHuSpe} to $\Sigma\subset X$ for any foliated symplectic surface $(X,\Omega,\cf)$.

Given a normalised holomorphic differential such as $\omega_i$, its cohomology class $[\omega_i]$ gives rise to a vector field on $\cb$.  It corresponds to the vector field $\frac{\partial}{\partial z^i}$ with respect to the coordinates $z^1,...,z^g$.  This maps to a vector field, $[\omega_i]=\iota^*[\omega_i]$ which maintains the same name by abuse of notation, on the formal neighbourhood $\widehat{\cb}_{[\Sigma]}$.  It is simply given by the Taylor expansion of $[\omega_i]$ at $[\Sigma]$.  Above $[\omega_i]$ is a vector field $\widehat{\omega}_i=D\theta\left(\frac{\partial}{\partial z^i}\right)$ on $\cl_{\text{KS}}$: 
\begin{equation}  \label{holdifser}
\widehat{\omega}_i=\omega_i-z^j\oint_{\hat{b}_i}\oint_{\hat{b}_j}\omega_{0,3}-\frac{1}{2}z^jz^k\oint_{\hat{b}_i}\oint_{\hat{b}_j}\oint_{\hat{b}_k}\omega_{0,4}-...
\end{equation} 
Again the series for  $[\widehat{\omega}_i]=[\omega_i]\in\bc^{2g} \llbracket z^1,...,z^g \rrbracket $ is analytic in $z^1,...,z^g$ in contrast to the formal series for $\widehat{\omega}_i$.  It gives the analytic expansion of a holomorphic section of the bundle   
\[
\begin{tikzcd}
\ch\arrow{d}\\ \cb\arrow[bend left]{u}{[\widehat{\omega_i}]}
 \end{tikzcd}
 \]
which takes the cohomology class of the holomorphic differential over each $[\Sigma]\in\cb$ normalised to have constant $a$-periods.  The analytic expansion is:
\[\oint_{b_i}\widehat{\omega_j}=-\sum_{I}\frac{z^I}{|I|!}\oint_{b_i}\oint_{\hat{b}_j}\oint_{\hat{b}_I}\omega_{0,|I|+2}=\tau_{ij}-z^k\oint_{b_i}\oint_{\hat{b}_j}\oint_{\hat{b}_k}\omega_{0,3}-...
\]
where $\oint_{\hat{b}_I}:=\oint_{\hat{b}_{i_1}}...\oint_{\hat{b}_{i_n}}$ for $I=(i_1,...,i_n)$.

For any $m\geq 0$, the series \[T_m(z^1,...,z^g)=\displaystyle\sum_{I}\frac{z^I}{|I|!}\oint_{\hat{b}_I}\omega_{0,|I|+m}\]
is defined in the $k$th formal neighbourhood of $[\Sigma]\in\cb$.  Its cohomology class is denoted $[T]$.  We summarise the geometric meaning of $T_m$ for small values of $m$ here:
\[
\begin{array}{rcll}
F_0&=& \displaystyle -\sum_{I}\frac{z^I}{|I|!}\oint_{\hat{b}_I}\omega_{0,|I|},&\text{Prepotential}\\
\theta&=& \displaystyle -\sum_{I}\frac{z^I}{|I|!}\oint_{\hat{b}_I}\omega_{0,|I|+1},&[\theta]=\text{Cohomology class defined in \eqref{cohtheta}}\\
\widehat{\mathcal{T}}&=& \displaystyle \sum_{I}\frac{z^I}{|I|!}\oint_{\hat{b}_I}\omega_{0,|I|+2},&[\widehat{\mathcal{T}}]=\tau_{ij}\\
\widehat{A}&=&\displaystyle\sum_{I}\frac{z^I}{|I|!}\oint_{\hat{b}_I}\omega_{0,|I|+3},&[\widehat{A}]=\text{Donagi-Markman cubic.}
 \end{array}
 \]

\subsubsection{The connection $\nabla^\cf$.}
Given $\Sigma\subset (X,\Omega,\cf)$ and $\cb$ the deformation space of $\Sigma$ in $X$, let $Z\stackrel{\pi}{\to}\cb$ be the universal family of curves in $\cb$, which comes with a natural map $Z\to X$ which induces the map $\Sigma\to X$ on each fibre of $Z$ over $[\Sigma]\in\cb$.  The fibres of $Z\to\cb$ induce a one-dimensional foliation which we call vertical.

The codimension one foliation $\cf$ on $X$ induces a codimension one foliation $H$ on $Z$ denoted by $H_z\subset T_zZ$ for $z\in Z$.   It satisfies
\begin{equation}  \label{horcon}
T_zZ\cong T_z\Sigma\oplus H_z,\qquad z\in Z-Z_R,
\end{equation} 
where $[\Sigma]=\pi(z)$ and $Z_R\subset Z$ is the codimension one set where the foliation intersects the vertical foliation non-transversally. On $Z^*=Z-Z_R$, $H_z\cong T_{[\Sigma]}\cb$ and \eqref{horcon} defines a horizontal lift of $T_{[\Sigma]}\cb$ to $T_zZ$.
\begin{definition}
Define a connection $\nabla^\cf$ on $Z^*\to\cb$ by the splitting \eqref{horcon}.
\end{definition}
The connection $\nabla^\cf$ is flat since leaves of the foliation $H$ give local flat sections.  This connection appears in many places, often implicitly, for families of varieties such as Hurwitz spaces \cite{DubGeo,DNOPSPri}, Seiberg-Witten families of curves \cite{NOkSei}, the Rauch variational formula in Appendix~\ref{variation} and in the work of Eynard and Orantin \cite{eynard_orantin,KSoAir,}. 
The connection lifts any vector in $T_{[\Sigma]}\cb$ to a vector in $T_zZ$ for $z\in Z^*$ which is used to take a Lie derivative of any tensor defined on $Z$ such as a locally defined relative differential $\eta$.  We can allow $\eta$ to be meromorphic.  This can be achieved either by considering $\nabla^\cf\eta$ locally on patches where $\eta$ is holomorphic and gluing, or by replacing $\eta$ with $\eta\cdot(u-\lambda)^m$, where $u$ is a locally defined function on $X$ that defines the foliation, and $\lambda$ is a function on $\cb$, chosen so that $\eta\cdot(u-\lambda)^m$ is holomorphic for $m$ large enough.  Then 
\[\nabla_v^\cf\big(\eta\cdot(u-\lambda)^m\big)=\nabla_v^\cf(\eta)\cdot(u-\lambda)^m-hm(u-\lambda)^{m-1}v\cdot\lambda\] 
which defines $\nabla_v^\cf\eta$ in terms of the covariant derivative of local holomorphic functions.  The covariant derivative $\nabla_v^\cf\eta$ naturally has poles, so the construction above allows one to take multiple covariant derivatives.   In particular it defines a covariant derivative on sections of the bundle $\bw=\cb\times W\to\cb$.  For an open neighbourhood $U_\Sigma\subset\cb$ of $[\Sigma]\in\cb$, and for each $v\in T_{[\Sigma]}\cb$ 
\[\nabla^\cf_v:\Gamma(U_\Sigma,\bw)\to\Gamma(U_\Sigma,\bw),
\]
which leaves $\Gamma(U_\Sigma,\bg)\subset\Gamma(U_\Sigma,\bw)$ invariant.

For any closed contour $\gamma\subset\Sigma_0$,
\begin{equation}  \label{concont}
\frac{\partial}{\partial z^i}\oint_\gamma\eta=\oint_\gamma\nabla^\cf_i\eta,
\end{equation} 
where $\nabla^\cf_i=\nabla^\cf_{\frac{\partial}{\partial z^i}}$.   To prove \eqref{concont}, define a local coordinate $x$ on $Z$ chosen so that $x=$ constant defines the foliation $H$ induced by $\cf$.  Express $\eta$ in terms of the local coordinate $x$ and depending on parameters $z^i$, differentiate under the integral sign, since the contour is compact, and use $\nabla^\cf u=0$.

In particular, for $\eta\in\Gamma(U_\Sigma,\bg)$, its cohomology class $[\eta]\in\Gamma(U_\Sigma,\ch)$, is determined by its periods hence \eqref{concont} implies that $\nabla^\cf$ lives above the Gauss-Manin connection:
\begin{equation}  \label{connFGM}
[\nabla^\cf\eta]=\nabla^{\text{GM}}[\eta].
\end{equation}
This restricts to formal neighbourhoods to give
\[\left[\nabla^\cf\widehat{\omega}_j\right]=\nabla^{\text{GM}}\left[\omega_j\right]
\]
for $\widehat{\omega}_j$ defined in \eqref{holdifser}.  It is shown in Section~\ref{geomA} that the covariant derivative $\nabla_i^\cf\widehat{\omega}_j$ gives rise to the tensor $A_\Sigma\in V_\Sigma\otimes V_\Sigma\otimes V_\Sigma$, and since $\nabla_i^{GM}\left[\omega_j\right]$ gives rise to the tensor $\bar{A}_\Sigma\in\overline{V}_\Sigma\otimes \overline{V}_\Sigma\otimes \overline{V}_\Sigma$, the compatibility of the covariant derivatives of vector fields is used to
prove that $A_\Sigma\mapsto\bar{A}_\Sigma$ under the map $V_\Sigma\to\overline{V}_\Sigma$.  This is proven in Proposition~\ref{vfs}.

Note that parallel transport, hence a flat frame, on $Z$ (or $\bw$ or $\bg$) for $\nabla^\cf$ does not exist in general due to the non-existence of solutions to ODEs at points where the foliation does not meet the vertical fibres transversally.  However, it does exist on any formal neighbourhood of a point $[\Sigma]\in\cb$.  
\begin{example}
Define a foliated surface locally by the family parametrised by  $z$
\[x=y^2+z.\]
Leaves of the foliation $x=\mathrm{constant}$ defines a flat connection on the fibration defined by the family.  Consider parallel transport from a general fibre to the fibre over $z=0$ given by $x=y_0^2$.  We have
\[y(y_0)=\sqrt{y_0^2-z}=y_0\sqrt{1-z/y_0^2}=y_0\left(1-\frac{z}{2y_0^2}-\frac18\frac{z^2}{y_0^4}-...\right)\]
exists analytically only when $|y_0|>|z|$ whereas it exists in $\bc[z]/z^n$ for any $n$.  For example,
in $\bc[z]/z^2$, $y(y_0)=y_0-\dfrac{z}{2y_0}$ which defines a path
\begin{align*}
\bc-\{0\}&\to\bc\\
y_0&\mapsto y_0-\frac{z}{2y_0}
\end{align*}
giving parallel transport above the first formal neighbourhood.  
\end{example}

\subsection{Formal and convergent series}  

The series given in  \eqref{thetaTR} and \eqref{holdifser} are induced via the natural map between a formal neighbourhood $\widehat{\cb}_{[\Sigma]}$ and an actual neighbourhood $U_\Sigma\subset\cb$ of $[\Sigma]\in\cb$
\[\widehat{\cb}_{[\Sigma]}\stackrel{\iota}{\longrightarrow} U_\Sigma\to\cb.
\]
The restriction $\iota^*$ sends locally defined functions to formal series.  It naturally extends from locally defined functions to locally defined sections such as (relative) meromorphic differentials.  
On $\bg\to\cb$, the holomorphic bundle with fibre $G_\Sigma\subset W$ over $[\Sigma]\in\cb$, it defines  
\[\iota^*:\Gamma(U_\Sigma, \bg)\to\Gamma(\widehat{\cb}_{[\Sigma]}, \bg)\]
which associates to any local section of the bundle $\bg\to\cb$ a section of the bundle $\bg$ defined in the $k$th formal neighbourhood of $[\Sigma]\in\cb$ for each $k$.   This map depends on the choice of connection $\nabla^\cf$.  It expresses a section of $\bg$ in terms of a flat frame for $\bg$ over the $k$th formal neighbourhood at $[\Sigma]$, for any $k$. 
\begin{lemma}
For $\eta\in\Gamma(U_\Sigma, \bg)$, with respect to local coordinates $\{z^1,...,z^g\}$ defined on $U_\Sigma\subset\cb$
\begin{equation}  \label{expansion}
  \iota^*\eta=\sum_{I}\frac{z^I}{|I|!}\left(\nabla^\cf_I\eta\right)|_\Sigma,
\end{equation}
on the $k$th formal neighbourhood of $[\Sigma]\in\cb$ for each $k$.  The sum is over tuples of positive integers $I=(i_1,...,i_n)\in\{1,...,g\}^n$, $z^I=\prod z^{i_k}$, $|I|=\sum i_k$, $\nabla^\cf_i=\nabla^\cf_{\partial/\partial z^i}$ and $\nabla^\cf_I=\nabla^\cf_{i_1}\cdots\nabla^\cf_{i_n}$.   
\end{lemma} 
\begin{proof}
The formula \eqref{expansion} is essentially a Taylor series for $\eta$.  We need to explain the appearance of the covariant derivative.

Consider the universal family $\pi:Z\to\cb$ and an open neighbourhood $V_Z\subset\pi^{-1}(U_\Sigma)$.  Then $\co_Z(V_Z)$ is a module over $\co_\cb(U_\Sigma)$. For any locally defined function $h\in\co_Z(U_Z)$, the restriction of $h$ to a formal neighbourhood of the fibre $\Sigma$ of $Z$ above $[\Sigma]\in\cb$ is defined by  
\begin{equation}  \label{expansionh}
  \iota^*h=\sum_{I}\frac{z^I}{|I|!}\left(\nabla^\cf_Ih\right)|_\Sigma.
\end{equation}
The functions $z^i\in\co_\cb(U_\Sigma)$ pull back to functions $z^i\in\co_Z(V_Z)$, with the same name by abuse of notation.  The Taylor series for $h\in\co_Z(U_Z)$ would normally use partial derivatives with respect to the vector fields $\frac{\partial}{\partial z^i}$, but these are not yet defined on $Z$ until a full system of coordinates is defined. Choose a locally defined function $u$ on $Z$ which induces the foliation on $Z$.  The collection $\{u,z^1,...,z^g\}$ defines local coordinates on $Z$ (when the foliation on $Z$ meets the fibre transversally).  The coordinates give rise to well-defined vector fields $\frac{\partial}{\partial z^i}$ on $Z$ (given the same name as vector fields on $\cb$) which allow one to write out a Taylor series with $\nabla^\cf_i$ replaced by $\frac{\partial}{\partial z^i}$ in \eqref{expansionh}.    The vector fields $\frac{\partial}{\partial z^i}$ are independent of a change of coordinates $\{u,z^1,...,z^g\}\mapsto\{f(u),z^1,...,z^g\}$ hence the Taylor series depends only on the foliation $\cf$ (which induces the folation on $Z$).  The use of the covariant derivative in \eqref{expansionh} to signify the choice of local coordinate $u$ is natural since $\nabla^\cf_iu=0$ agrees with the definition of the vector fields from coordinates via $\frac{\partial}{\partial z^i}u$.  Note that a more general change of coordinates $\{u,z^1,...,z^g\}\mapsto\{f(u,z^1,...,z^g),z^1,...,z^g\}$ which is equivalent to a different choice of connection unrelated to the foliation, leads to a different right hand side in \eqref{expansionh}, where the relation between the two different series is achieved via $z^i$-dependent coefficients in \eqref{expansionh}

The formula \eqref{expansion} follows from \eqref{expansionh} since $\bg$ is the push-forward of the sheaf of relative differentials on $Z$, hence $\eta\in\Gamma(U_\Sigma, \bg)$ is built locally from  $h\in\co_Z(U_Z)$.
\end{proof}


On the universal family $Z\to\cb$, let $h\in\co_Z(U_Z)$ be a locally defined function on an open set $U_Z\subset Z$.  Its image in a formal neighbourhood of a fibre $\Sigma$ uses the same formula as \eqref{expansion}---the difference between functions and differentials is minor since $hdu$ is a locally defined differential and $\nabla^\cf du=0$.  

The ring homomorphism $\iota^*$ in \eqref{expansionh} satisfies $\iota^*{1}=1$ which is visible on the right hand side of \eqref{expansionh} since $\nabla^\cf1=0$.  The equality $\iota^*(h_1h_2)=\iota^*(h_1)\iota^*(h_2)$ is the combinatorial identity
\[\sum_{I}\frac{z^I}{|I|!}\nabla^\cf_I(h_1h_2)|_\Sigma=\sum_{I}\frac{z^I}{|I|!}\nabla^\cf_I(h_1)|_\Sigma\sum_{I}\frac{z^I}{|I|!}\nabla^\cf_I(h_2)|_\Sigma\]
which follows from Leibniz' formula applied to $\nabla^\cf$.  It formally coincides with the identity showing that the Taylor expansion in several variables of a product of two functions is the product of the two Taylor expansions.  

Covariant differentiation $\nabla^\cf_i$ on the formal neighbourhood $\widehat{\cb}_{[\Sigma]}$ is simply given by $\frac{\partial}{\partial z^i}$ for each $i=1,...,g$.  The map $\iota^*$ commutes with $\nabla^\cf$:
\[\nabla^\cf\circ\iota^*=\iota^*\circ\nabla^\cf.\]
The proof that $\frac{\partial}{\partial z^i}\iota^*(h)=\iota^*(\nabla_i^\cf h)$ is combinatoric and formally coincides with differentiation of a Taylor expansion in several variables.

\begin{proposition}   \label{thetader}
The section $\theta\in\Gamma(\widehat{\cb}_{[\Sigma]},G_\Sigma)$ defined in \eqref{thetaTR} satisfies the following relation with respect to local FD coordinates $(u,v)$ on $X$:
\begin{equation}  \label{eq:thetader}
    \iota^*(vdu)=\sum_{|I|\geq0}\frac{z^I}{|I|!}\nabla^\cf_I(vdu)|_\Sigma=vdu|_\Sigma-\theta.
\end{equation}
\end{proposition}
\begin{proof}
Given local FD coordinates $(u,v)$ around $\alpha\in R\subset\Sigma\subset X$, the local differential $vdu$ pulls back to a local relative differential on the universal space $Z$.
Hence we can apply \eqref{expansion} to get the first equality in \eqref{eq:thetader}.  Write \eqref{eq:thetader} as 
\begin{equation}  \label{eq:thetader1}
\iota^*(vdu)=vdu|_\Sigma-\xi
\end{equation}
so it remains to prove that  $\xi=\theta$.
Note that $\xi$ is invariant under a change of local FD coordinates $(u,v)\mapsto(f(u),v/f'(u)+g(u))$ since $\nabla^\cf_i(g(u)df(u))=0.$

Each holomorphic differential $\eta\in H^0(\Sigma,K_\Sigma)$ extends to a family of normalised holomorphic differentials $\widetilde{\eta}([\Sigma'])\in H^0(\Sigma', K_{\Sigma'})$ for $[\Sigma']\in U_\Sigma$ by requiring that the $a$-periods are constant, for example $\oint_{a_j}\widetilde{\omega}_i=\delta_{ij}$, $i,j=1,...,g$.  We write
\[\nabla_I^\cf\eta:=\left(\nabla_I^\cf\widetilde{\eta}\right)|_\Sigma\in H^0(K_\Sigma(mR)),\quad m=2|I|
\]
where, as usual, $\nabla_I^\cf=\nabla_{i_1}^\cf...\nabla_{i_n}^\cf$. 

By definition, although $vdu$ is locally defined on the fibre $\Sigma$, its covariant derivative is globally defined:
\begin{equation}  \label{holder}
\nabla^\cf_i(vdu)=-\widetilde{\omega}_i.
\end{equation} 
From \eqref{eq:thetader1} and \eqref{holder} we have
\[\xi=z^i\omega_i+\frac12z^iz^j\nabla^\cf_i\omega_j+\frac{1}{3!}z^iz^jz^k\nabla^\cf_i\nabla^\cf_j(\omega_k)+...
\]
where $\nabla^\cf_I\omega_k\in H^0(K_\Sigma(mR))$ for $m=2|I|$. 

We have $\oint_{\hat{b}_i}\omega_{0,2}=-\omega_i$, and the correlators satisfy the following variational formula due to Eynard and Orantin \cite{eynard_orantin}:
\begin{equation}   \label{EOvar}
  \nabla^{\mathcal{F}}_i\omega_{h,n}(p_1,\cdots,p_n) = \oint_{p_{n+1}\in \hat{b}_i}\omega_{h,n+1}(p_1,\cdots,p_n,p_{n+1}),\quad n>0.
\end{equation}
The formula \eqref{EOvar} is proven in Appendix~\ref{variation} inductively.  Applied to $h=0$, we have
$$\nabla_i^\cf\omega_{0,n}=\oint_{\hat{b}_i}\omega_{0,n+1},\quad n\geq 2
$$
hence
$$\nabla_I^\cf\omega_i=-\oint_{\hat{b}_i}\oint_{\hat{b}_I}\omega_{0,|I|+2}.
$$
Thus $\xi=\theta$ and the proposition is proven.
\end{proof}

\begin{proposition}   \label{phisec1}
The section $\theta\in\Gamma(\widehat{\cb}_{[\Sigma]},G_\Sigma)$ defines a map
$$\theta:\widehat{\cb}_{[\Sigma]}\to\cl_{\text{KS}},
$$
where $\cl_{\text{KS}}\subset W$ is the quadratic Lagrangian defined in \eqref{rescon1} and \eqref{rescon2}.
\end{proposition}
\begin{proof}
We need to check \eqref{rescon1} and \eqref{rescon2} for $\eta=\theta$ with respect to the FD coordinates $(u_\alpha,v_\alpha)$ from Definition~\ref{specoord}.
\begin{align*}
\Res_\alpha\left(v_\alpha-\frac{\theta}{du_\alpha}\right)u_\alpha^mdu_\alpha&=\Res\left(v_\alpha u_\alpha^mdu_\alpha-\theta u_\alpha^m\right)\\
&=\Res_\alpha\left(v_\alpha u_\alpha^mdu_\alpha+\sum_{|I|>0}\frac{z^I}{|I|!}u_\alpha^m\nabla^\cf_I(v_\alpha du_\alpha) \right)\\
&=\Res_\alpha\left(v_\alpha u_\alpha^mdu_\alpha+\sum_{|I|>0}\frac{z^I}{|I|!}\nabla^\cf_I(v_\alpha u_\alpha^mdu_\alpha) \right)\\
&=0
\end{align*}
for any $m\geq 0$ where the second equality uses \eqref{eq:thetader}.
 The final equality uses the fact that $v_\alpha u_\alpha^mdu_\alpha$ is holomorphic at $\alpha$ hence has zero residue.  Furthermore, it has zero residue in a neighbourhood so its higher derivatives also vanish to give
 \begin{equation} \label{vanres}
     \Res_\alpha\sum_{|I|>0}\frac{z^I}{|I|!}\nabla^\cf_I(v_\alpha u_\alpha^mdu_\alpha)=\sum_{|I|>0}\frac{z^I}{|I|!}\frac{\partial}{\partial z^I}\Res_\alpha(v_\alpha u_\alpha^mdu_\alpha)=0.
 \end{equation} 
Hence $\theta$ satisfies the first of the residue constraints \eqref{rescon1}.

Let $p$ be any locally analytic function.   A consequence of the ring homomorphism property is:
$$\iota^*p(v_\alpha)=p(\iota^*(v_\alpha))=p\left(v_\alpha-\frac{\theta}{du_\alpha}\right).
$$
In particular, 
\[\left(v_\alpha-\frac{\theta}{du_\alpha}\right)^2u_\alpha^mdu_\alpha=\iota^*(v_\alpha)^2u_\alpha^mdu_\alpha=\sum_{I}\frac{z^I}{|I|!}\nabla^\cf_I(v_\alpha^2 u_\alpha^mdu_\alpha)\]
and
$$
\Res_\alpha\left(v_\alpha-\frac{\theta}{du_\alpha}\right)^2u_\alpha^mdu_\alpha=\sum_{I}\frac{z^I}{|I|!}\frac{\partial}{\partial z^I}\Res_\alpha(y^2_\alpha u_\alpha^mdu_\alpha)=0
$$
since $v_\alpha^2 u_\alpha^mdu_\alpha$ is holomorphic in a neighbourhood of $\alpha\in\Sigma_0\subset X$.  Hence $\theta$ also satisfies the second of the residue constraints \eqref{rescon2} and the proposition is proven.
\end{proof}
\begin{remark}
In the proof of Proposition~\ref{phisec1} it is shown that $\theta$ satisfies
$$
\Res_\alpha\left(v_\alpha-\frac{\theta}{du_\alpha}\right)^ku_\alpha^mdu_\alpha=0
$$
for $k=1,2$ and for all $m\geq 0$.  The proof easily generalises to allow all $k\geq 1$.
\end{remark}
Proposition~\ref{phisec1} yields a collection of relations among periods and residues of $\omega_{h,n}$.  Here we list a few of them.  We have \[0=\Res_\alpha\left(v_\alpha-\frac{\theta}{du_\alpha}\right)u_\alpha^mdu_\alpha=-\Res_\alpha\theta u_\alpha^m=\sum_I\frac{z^I}{|I|!}\Res_\alpha u_\alpha^m\oint_{\hat{b}_I}\omega_{0,|I|+1} \]
which implies that the principal part of $\oint_{\hat{b}_I}\omega_{0,|I|+1}$ is skew-invariant under the local involution defined by the $\cf$.  The quadratic relation
\[0=\Res_\alpha\left(v_\alpha-\frac{\theta}{du_\alpha}\right)^2u_\alpha^mdu_\alpha=-2\Res_\alpha v_\alpha\theta u_\alpha^m+\frac{\theta\cdot\theta u_\alpha^m}{du_\alpha} \]
yields a sequence of relations.  When $m=0$, the first two relations are: 
\begin{equation}  \label{relat}
    \Res_\alpha\frac{\omega_i\omega_j}{du_\alpha}=-\Res_\alpha v_\alpha\oint_{\hat{b}_i}\oint_{\hat{b}_j}\omega_{0,3}
\end{equation}
and
\begin{align}  \label{relat2}
    \Res_\alpha\Big(\frac{\omega_i}{du_\alpha}\oint_{\hat{b}_j}\oint_{\hat{b}_k}\omega_{0,3}+\frac{\omega_j}{du_\alpha}\oint_{\hat{b}_k}\oint_{\hat{b}_i}\omega_{0,3}&+\frac{\omega_k}{du_\alpha}\oint_{\hat{b}_i}\oint_{\hat{b}_j}\omega_{0,3}\Big)\\
    &=-\frac13\Res_\alpha v_\alpha\oint_{\hat{b}_i}\oint_{\hat{b}_j}\oint_{\hat{b}_k}\omega_{0,4}.\nonumber
\end{align}

\begin{proposition}  \label{phisec2}
The cohomology class of $\theta\in\Gamma(\widehat{\cb}_{[\Sigma]},G_\Sigma)$ defined in \eqref{thetaTR} is the local section  $[\theta]\in \Gamma(U_\Sigma,\ch)$ defined in \eqref{cohtheta}. 
\end{proposition} 
\begin{proof}
The symplectic form $\Omega_X=-d(v_\alpha du_\alpha)$ on $X$ defines a bundle-valued 1-form $\omega_i\otimes dz^i:T_{[\Sigma]}\cb\to H^0(\Sigma,K_\Sigma)$ which is a section of $\Omega_\cb^1\otimes\bg$ that lives over the 1-form $\phi\in\Gamma(\Omega_\cb^1\otimes\ch)$ via the quotient $\bg\to\ch$.  Hence $\nabla^\cf(v_\alpha du_\alpha)=\omega_i\otimes dz^i$ lives over $\nabla^{\text GM}s=\phi$ where $s\in\Gamma(U_\Sigma,\ch)$ defines
$[\theta]([\Sigma']):=s([\Sigma'])-s([\Sigma])\in\bc^{2g}\cong \ch_\Sigma$.

By \eqref{connFGM}, $\nabla^\cf$ lives above $\nabla^{\text GM}$ so each higher covariant derivative $\nabla^\cf_I(v_\alpha du_\alpha)$ lives over the cohomology class $\nabla^{\text GM}_I[\theta]$.  Hence by \eqref{eq:thetader}, the series $\theta$ lives above the Taylor series for $[\theta]$ which completes the proof.
\end{proof}

\subsubsection{Higher genus.}
For $g>0$, analogous to \eqref{thetaTR} define  $\theta_g\in\Gamma(\widehat{\cb}_{[\Sigma]},G_\Sigma)$ by
\[  
\theta_g=\omega_{g,1}+z^i\oint_{\hat{b}_i}\omega_{g,2}+\frac12z^iz^j\oint_{\hat{b}_i}\oint_{\hat{b}_j}\omega_{g,3}+\frac{1}{3!}z^iz^jz^k\oint_{\hat{b}_i}\oint_{\hat{b}_j}\oint_{\hat{b}_k}\omega_{g,4}-...
\]
Then the cohomology class $[\theta_g]\in \Gamma(\widehat{\cb}_{[\Sigma]},\ch)$ is analytic in $z^1,...,z^g$ and coincides with the analytic expansion of $[\omega_{g,1}]$ due to the following lemma which generalises \eqref{EOvar} to the case $n=0$.
\begin{lemma} \label{varF}
For $h \geq 2$, the function $F_h$ defined in \eqref{hatf} satisfies the relation 
$$\frac{\partial F_h}{\partial z^i}=\oint_{b_i}\omega_{h,1}.$$
\end{lemma}
\begin{proof}
The proof of \eqref{EOvar} uses \eqref{rec} which is not available in the case of $n=0$.  Instead we must use the
definition of $F_h$ given in \eqref{hatf} for $h>1$  by
\[
F_h=\frac{1}{2h-2}\sum_{du(\alpha) = 0}\mathop{\text{Res}}_{p=\alpha}\psi(p)\omega_{h,1}(p)
\]
where $d\psi=vdu$.

Note that since $\nabla^\cf_id\psi=-\omega_i$ then
$\nabla^\cf_i\psi=-f_i$ where $f_i$ is a primitive of the holomorphic differential $\omega_i$ on $\Sigma-\{a_i,b_i\}$, i.e. $df_i=\omega_i$. Importantly, although $d\psi=vdu$ is only locally defined on $\Sigma$, its variation can be represented by a global holomorphic differential which allows us to take periods along global cycles in $\Sigma$.  Then:
\begin{align*}
(2h-2)\frac{\partial F_h}{\partial z^i}&=\frac{\partial }{\partial z^i}\sum\Res_{p=\alpha}\psi(p)\omega_{h,1}(p)\\
&=\sum\Res_{p=\alpha}\left[\left(\nabla^\cf_i\psi(p)\right)\omega_{h,1}(p)+\psi(p)\left(\nabla^\cf_i\omega_{h,1}(p)\right)\right]\\
&=\sum\Res_{p=\alpha}\left[f_i(p)\omega_{h,1}(p)+\psi(p)\oint_{b_2}\omega_{h,2}(p,p')\right]\\
&=(2h-1)\oint_{b_2}\omega_{h,1}(p')+\sum\Res_{p=\alpha}f_i(p)\omega_{h,1}(p)\\
=(2h-1)&\oint_{b_2}\omega_{h,1}(p')+\sum_{j=1}^g\left(\oint_{a_j}\omega_{h,1}(p)\oint_{b_j}\omega_i(p)-\oint_{b_j}\omega_{h,1}(p)\oint_{a_j}\omega_i(p)\right)\\
=(2h-2)&\oint_{b_2}\omega_{h,1}(p')
\end{align*}
where the third equality uses \eqref{EOvar} and the final two equalities use the Riemann bilinear relations and vanishing of $a$-periods of $\omega_{h,1}$.
\end{proof}
Together with the variational formula, Lemma~\ref{varF} implies the relation
$$\partial_{i_1} \dots \partial_{i_n} F_h=\int_{p_1 \in b_{i_1}} \dots \int_{p_n \in b_{i_n}} \omega_{h,n}(p_1, \dots, p_n).$$

Just as the symmetry of derivatives of the periods of $\theta$, given by $\frac{\partial}{\partial z_i}\w_j=\tau_{ij}$ leads to the potential $F_0$, the same mechanism yields $F_1$.
Applied to $(h,n)=(1,1)$, the variational formula yields
$$\frac{\partial}{\partial z^j}\oint_{b_i}\omega_{1,1}=\oint_{b_i}\oint_{b_j}\omega_{1,2}=\frac{\partial}{\partial z^i}\oint_{b_j}\omega_{1,1}
$$
since $\omega_{1,2}$ is symmetric.  Hence there exists a potential $F_1$ defined up to a constant by
$$\frac{\partial F_1}{\partial z^i}=\oint_{b_i}\omega_{1,1}.$$
We have seen that $F_h$ is defined via a variational formula and via topological recursion together with the dilaton equation for $h\geq 2$.  These definitions are fundamentally different since the variational approach requires knowledge of $F_h$ in a neighbourhood $U\subset\cb$ whereas the topological recursion definition requires only knowledge at a point $b\in\cb$.


  
\subsection{Geometry of the tensor \texorpdfstring{$A_\Sigma$}{A} }  \label{geomA}

The Lagrangian $\cl_{\text{KS}}\subset W$ is defined in a formal neighbourhood of $0\in W$.  A vector field on $\cl_{\text{KS}}$ is a derivation given by a linear combination of $\frac{\partial}{\partial x^i}$ and $\frac{\partial}{\partial y_i}$ with coefficients defined in a formal neighbourhood of $0\in W$.  We present here explicit formulae for vector fields on $\cl_{\text{KS}}$ and relate them to normalised holomorphic differentials $\omega_i$ and
$\widehat{\omega}_i\in\Gamma(\widehat{\cb}_{[\Sigma]},G_\Sigma)$ defined in \eqref{holdifser}.

Coordinates \(\{x^i\}\) on $\cl_{\text{KS}}$ are the restriction of those given in Definition~\ref{darbcoord}. Dual to \(\{dx^i\}\) are the following vector fields on $\cl_{\text{KS}}$
\begin{equation}  \label{vf}
\xi_i=\frac{\partial}{\partial x^i}+ f_{ij}\frac{\partial}{\partial y^j}
    =(0,...,1,...\mid f_{i\,1},f_{i\,2},...)
\end{equation}
defined in any formal neighbourhood of $0\in W$.  The coefficients $f_{ij}$ are functions of $\{x^k\}$ defined in each formal neighbourhood of $0\in W$.  They satisfy the linear system:
    \begin{align*} 
    0&=dH_i(\xi_j)=\xi_j(H_i)=\left(\frac{\partial}{\partial x^j}+f_{jk}\frac{\partial}{\partial y^k}\right)H_i\\
    &=2a_{ijk}x^k+2b_{ij}^ky_k+ f_{jk}(-\delta_{ik}+2b_{i\ell}^kx^\ell+2c_i^{\ell k}y_\ell).\nonumber
    \end{align*} 
    They can be calculated in the $k$th formal neighbourhood of $0\in W$, for any $k$, using the recursive procedure described in \eqref{graph}, hence expressing $\cl_{\text{KS}}$ as the image of $S(x)= (x,y(x))$.  Linearise this to produce 
$$\xi_i=DS\left(\frac{\partial}{\partial x^i}\right)\quad\Rightarrow\quad f_{ij}=2a_{ijk}x^k+(4b_{jk}^ma_{i\ell m}+2b_{ji}^ma_{k\ell m})x^kx^\ell+...\ .$$

\begin{proposition}  \label{vfs}
The vector fields on $\cl_{\text{KS}}$ satisfy the following:
\begin{align}
\xi_i&=\widehat{\omega}_i,\qquad i=1,...,g  \label{xiomega}\\
\nabla^{\cf}_{\widehat{\omega}_i}\widehat{\omega}_j&\mapsto \nabla^{\text{GM}}_{[\omega_i]}[\omega_j]  \label{covFGM}\\
A_\Sigma&\mapsto\bar{A}_\Sigma.  \label{AtoA}
\end{align}
\end{proposition}
\begin{proof}
The map $h:H^0(\Sigma,K_\Sigma)\to L$, maps holomorphic differentials normalised over $a$-periods $\{\omega_i\mid i=1,...,g\}$, to its local expansion at $R\subset\Sigma$.
\(h\) has a natural image with respect to the coordinates $\{x^i,y_i\}$:
\begin{equation}  \label{hvec}
\frac{\partial}{\partial x^i}=h(\omega_i),\quad i=1,...,g.
\end{equation}
To see this note that $\{\frac{\partial}{\partial x^i}\}$ are dual to the differentials $\{dx^i\}$, so one needs to calculate the action of $h(\omega_i)$ on $dx^j$.  Since $x^j$ is linear, any vector acts by $\langle v,dx^j\rangle=v\cdot x^j=\Omega_W(v,x^j)$.  Now 
$$\Omega_W(\omega_i,x^j)=\sum_{k=1}^g\left(\oint_{a_k}\omega_i\oint_{\hat{b}_k}x^j-\oint_{\hat{b}_k}\omega_i\oint_{a_k}x^j\right)=\delta_{ij}
$$
proving $h(\omega_i)=\frac{\partial}{\partial x^i}$, $i=1,...,g$.  Since the map $h$ coincides with the linearisation of the section $\theta\in\Gamma(\widehat{\cb}_{[\Sigma]},G_\Sigma)$ evaluated at the point $[\Sigma]\in\cb$, \eqref{hvec} is the specialisation of \eqref{xiomega} to the 1st formal neighbourhood.

The functions $z^i$ on $\ch_\Sigma$ and $x^i$ on $W$ are related as follows.  Under the symplectic quotient, $z^i$ maps to $x^i|_{G_{\Sigma_0}}$ for $i=1,...,g$ since for $\eta\in G_{\Sigma}$, by the Riemann bilinear relations $\langle x^i,\eta\rangle=\oint_{a_i}\eta=z^i([\eta])$.  In a formal neighbourhood of $[\Sigma]\in\cb$ the linearisation $D\theta$ sends the vector field $\frac{\partial}{\partial z^i}$, defined on $\cb$ and hence on the formal neighbourhood of $[\Sigma]\in\cb$, to
$$\frac{\partial\theta}{\partial z^i}:=\widehat{\omega_i}=\sum_{I}\frac{z^I}{|I|!}\oint_{\hat{b}_i}\oint_{\hat{b}_I}\omega_{0,|I|+2}=\omega_i+z^j\oint_{\hat{b}_i}\oint_{\hat{b}_j}\omega_{0,3}+...\ .
$$
Hence the image of $\widehat{\omega_i}$ is obtained by replacing $z^I$ by $x^I$ in $\theta$ to give
$$\widehat{\omega}_i|_{z^j=x^j}=\xi_i|_{G_{\Sigma}}
$$
which is \eqref{xiomega}.  The first two terms of \eqref{xiomega} are
$$\xi_i=\frac{\partial}{\partial x^i}+f_{ij}\frac{\partial}{\partial y^j}
$$
for $f_{ij}=a_{ijk}x^k+...$
the first terms $\omega_i$ and $\frac{\partial}{\partial x^i}$ agree by \eqref{hvec} and the second terms $x^j\oint_{b_i}\oint_{b_j}\omega_{0,3}$ and $a_{ijk}x^k\frac{\partial}{\partial y^j}$ also agree by the following.
  +
       
A variation of the vector field is given by
    $$\frac{\partial}{\partial x^j}\xi_i=(0,...,0,...\mid \frac{\partial}{\partial x^j}f_{i\,1},\frac{\partial}{\partial x^j}f_{i\,2},...)\in V_\Sigma.
    $$
Differentiate the expression for $f_{ij}$ by $x^k$ and take the constant term to get 
    $$\frac{\partial}{\partial x^k}f_{ji}=2a_{ijk}.
    $$
    In other words, the tensor $A_\Sigma$ gives the map
    \[T_0\cl_{\text{KS}}\otimes T_0\cl_{\text{KS}}\to V_\Sigma\]
defined by variation of a vector field of $\cl_{\text{KS}}$ with respect to a vector.  It is a tensor because any vector in $T_0\cl_{\text{KS}}$ canonically extends to a vector field via \eqref{vf}.  The canonical isomorphism $T_0\cl_{\text{KS}}\cong V_\Sigma^*$ means that $A_\Sigma\in V_\Sigma\otimes V_\Sigma\otimes V_\Sigma$.


The Lagrangian $\cl_{\text{KS}}\subset W$ is a formal germ, and its vector fields are derivations on $W$ that annihilate the defining ideal of $\cl_{\text{KS}}$.  The tensor $A_\Sigma$ is defined via the covariant derivative $\nabla^{\cf}_uv$ of vector fields $v\in\Gamma(T\cl_{\text{KS}})$ by vectors $u\in L=T_0\cl_{\text{KS}}$ with respect to the flat connection $\nabla^{\cf}$ induced by the foliation $\cf$.  It defines a tensor on $L\otimes L$ because any vector $v\in L$ extends canonically to a vector field---$v$ is a linear combination of $\frac{\partial}{\partial x^i}$ which are mapped to $\frac{\partial}{\partial x^i}\mapsto \xi_i$ defined by \eqref{vf}.  
\end{proof}

An alternative, non-geometric proof of Theorem~\ref{main} can be obtained from Corollary~\ref{om3cor} combined with the following result.
\begin{proposition}[\cite{KSoAir}]   \label{A=w3}
Given $\Sigma\subset(X,\Omega,\cf)$, we have
\[A_\Sigma=\omega_{0,3}\in V_\Sigma\otimes V_\Sigma\otimes V_\Sigma.
\]
\end{proposition}
\begin{proof}
The element $\eta\in W$ lives in $\cl_{\text{KS}}$ if it satisfies the residue constraints \eqref{rescon1}. For $u_\alpha=z^2_\alpha$, $v_\alpha=z_\alpha$:
$$\Res\left(\frac{\eta}{du_\alpha}-v_\alpha\right)u_\alpha^mdu_\alpha=0,\quad m\geq 1,
$$
$$\Res\left(\frac{\eta}{du_\alpha}-v_\alpha\right)^2u_\alpha^mdu_\alpha=0,\quad m\geq 0.
$$
To analyse these we choose a new basis of $W$: \[\{x^{k,\alpha},y_{k,\alpha}\mid k\in\bn,\alpha\in R\}\]
where $x^{k,\alpha}$ has a pole of order $k$ at $\alpha\in\Sigma$ and is holomorphic on $\Sigma-\alpha$ and $y_{k,\alpha}=z_\alpha^k$ is defined only locally near $\alpha$ via the local coordinate $z_\alpha$.

The first residue constraint implies 
$$0=\Res_\alpha\left(\frac{\eta}{du_\alpha}-v_\alpha\right)u_\alpha^mdu_\alpha=\Res_\alpha\eta u_\alpha^m=-2m\langle y_{2m-1,\alpha},\eta\rangle
$$
where the last equality uses $u_\alpha=z_\alpha^2$ and $d(u_\alpha^m)=2my_{2m-1,\alpha}$.

The second implies 
\begin{align*}
    0&=\Res\left(\frac{\eta}{du_\alpha}-v_\alpha\right)^2u_\alpha^mdu_\alpha=\Res_\alpha\frac{\eta\cdot\eta }{du_\alpha}u_\alpha^m-2\Res_\alpha\eta z_\alpha u_\alpha^m\\
&=\Res_\alpha\frac{\eta\cdot\eta }{du_\alpha}u_\alpha^m-2\langle y_{2m,\alpha},\eta\rangle
\end{align*} 
which is a linear term $y_{2m,\alpha}$ plus a quadratic term
\[\Res_\alpha\frac{\eta\cdot\eta }{dx}x^m=a_{ijk}^{\alpha\beta\gamma}x^{j,\beta} x^{k,\gamma}+b_{ij\gamma}^{\alpha\beta k}x^{j,\beta} y_{k,\gamma}+c_{i\beta\gamma}^{\alpha jk}y_{j,\beta} y_{k,\gamma}.
\]
The right hand side is the most general quadratic term with respect to the coordinates $x^i_\beta=\langle x^i_\beta,\eta\rangle$ and $y_i^\beta=\langle y_i^\beta,\eta\rangle$.   The coefficients $a_*^*$, $b^*_*$ and $c_*^*$ are $m$-dependent.  To determine the coefficient of $x^j_\beta x^k_\gamma$ simply evaluate on any differential $\eta$ which is locally holomorphic since such $\eta$ annihilates $y_i$, i.e. $\langle y_i^\beta,\eta\rangle=0$.  When $\eta$ is locally holomorphic
\[\Res_\alpha\frac{\eta\cdot\eta }{dx}x^m=0,\quad \text{for }m>0
\]
since $x^m/dx=z^{2m-1}/dz$ has no pole, and nor does each factor of $\eta$.  Hence we are left with the case $m=0$
\[\Res_\alpha\frac{\eta\cdot\eta }{dx}=a_{ijk}^{\beta\gamma}x^j_\beta x^k_\gamma+...
\]
so $a_{ijk}^{\beta\gamma}=\frac14\delta_{ij}\delta_{ik}\delta_{\alpha\beta}\delta_{\alpha\gamma}$. Hence $A_\Sigma=\frac14\sum_\alpha x^1_\alpha\otimes x^1_\alpha\otimes x^1_\alpha$ which agrees with the following formula for $\omega_{0,3}$:
\[\omega_{0,3}(p_1,p_2,p_3)=\sum_\alpha\Res_{p=\alpha}\frac{B(p,p_1)B(p,p_2)B(p,p_3)}{du(p)dv(p)}.
\]
\end{proof}

\subsection{Analytical construction of $\theta$}
The section $[\theta]\in\Gamma(U_\Sigma,\ch)$ in \eqref{cohtheta} together with parallel transport by the Gauss-Manin connection $\nabla^{\text{GM}}$ on $\ch$ defines a local embedding
\[U_\Sigma\hookrightarrow\ch_\Sigma.
\]
The cohomology classes in $\ch_\Sigma$ are represented by meromorphic differentials on $\Sigma$ which is encoded by the surjective map $G_\Sigma\to\ch_\Sigma$.
The section $\theta\in\Gamma(\widehat{\cb}_{[\Sigma]},G_\Sigma)$ in \eqref{thetaTR} defines a map on a formal neighbourhood of $[\Sigma]\in\cb$. The failure to lift the embedding $U_\Sigma\hookrightarrow\ch_\Sigma$ to an embedding $U_\Sigma\to G_\Sigma$ is due to the failure of parallel transport for the connection $\nabla^\cf$.  Following Kontsevich and Soibelman \cite{KSoAir} one can regain parallel transport for the connection $\nabla^\cf$ on a bundle $\bg^0$ related to $\bg$.

Let $\Sigma\subset(X,\Omega,\cf)$ and  $U_\Sigma\subset\cb$ a ball neighbourhood of $[\Sigma]\in\cb$.  Choose a union of open balls in the universal space $D_R\subset Z$, containing the points $R\subset\Sigma\subset X$ where $\cf$ does not meet $\Sigma$ transversally, such that $D_R\cong U_\Sigma\times D^2$ and $U_\Sigma\times\partial D^2$ is tangent to the foliation on $Z$ induced by $\cf$.  The balls are chosen small enough that each component of $D_R$ contains a single point in $R$.
\begin{definition}
  Define the vector space
  \begin{equation} \label{gsigma0}
         G_\Sigma^0=\{ \eta\in H^0(\Omega^1(\Sigma-D_R))\mid\oint_\gamma\eta=0,\ \forall\text{ closed }\gamma\subset\partial(\Sigma- D_R)\}.
     \end{equation}
\end{definition}
Here, $\gamma\subset\partial(\Sigma- D_R)$ means that $\gamma\subset\Sigma- D_R$ and it is homotopic to a boundary component.
On the level of cohomology, $G_\Sigma^0$ behaves like $G_\Sigma$.  In particular, there is a surjective linear map $G_\Sigma^0\to\ch_\Sigma$ obtained by taking the cohomology class $\eta\mapsto[\eta]$. The vector space $G_\Sigma^0$ lives inside a weakly symplectic vector space $W^0$ of differentials defined in annuli around each point of $R\subset\Sigma$ with zero contour integrals around boundary circles.  Further details are in \cite{ChaSei}.

Define a bundle $\bg^0\to U_\Sigma$ with fibre over $[\Sigma']\in U_\Sigma$ given by $G^0_{\Sigma'}$ defined by \eqref{gsigma0} although using $D_R$ for $R\subset\Sigma$ rather than $\Sigma'$. The covariant derivative $\nabla^\cf$ acts on sections of $\bg^0$.  Parallel transport of $\nabla^\cf$ is well-defined on $\bg^0$ by construction.  On the bundle $\bg$, parallel transport is not defined due to the non-existence of solutions to the ODE at points where $\cf$ meets the curve $\Sigma'$ non-transversally, and $\bg^0$ simply removes those points.

Define $\theta^0\in\Gamma(U_\Sigma,\bg^0)$ analogously to the definition of $\theta$ in \eqref{eq:thetader}.
\[\theta^0([\Sigma'])=v_\alpha du_\alpha|_{\Sigma'}-g_\Gamma(v_\alpha du_\alpha)\]
where $g_\Gamma:G_{\Sigma}^0\to G_{\Sigma'}^0$ is defined by parallel transport with respect to $\nabla^\cf$ along a path $\Gamma\subset U_\Sigma$ joining $[\Sigma]$ and $[\Sigma']$.

The residue constraints (\ref{rescon1}) and (\ref{rescon2}) also make sense in the analytic setting and they define a quadratic Lagrangian $\cl^0_{\text{KS}}\subset W^0$. Choose local FD coordinates $(u_\alpha,v_\alpha)$ in $X$.  For any closed boundary component $\gamma\subset\partial (\Sigma-D_R)$ define $\cl^0_{\text{KS}}\subset W^0$ to consist of differentials $\eta\in W^0$ satisfying:
 \begin{alignat}{2}  \label{rescon10}
        &\oint_\gamma\left(v_\alpha-\frac{\eta}{du_\alpha}\right)u_\alpha^mdu_\alpha &= 0,\quad m\geq 1,\\
        &\oint_\gamma\left(v_\alpha-\frac{\eta}{du_\alpha}\right)^2u_\alpha^mdu_\alpha &= 0,\quad m\geq 0. \label{rescon20}
\end{alignat}

An analogue of Theorem~\ref{phisec} holds.
\begin{proposition}   \label{phisec0}
The section $\theta^0\in\Gamma(U_\Sigma,\bg^0)$ satisfies the following properties.
\begin{enumerate}
    \item It takes its values in $\cl^0_{\text{KS}}$.
    \item Its cohomology class $[\theta^0]\in \Gamma(U_\Sigma,\ch)$ coincides with $[\theta]$ defined in \eqref{cohtheta}.
\end{enumerate}
\end{proposition} 
\begin{proof}
\[\oint_\gamma\left(v_\alpha-\frac{\theta^0}{du_\alpha}\right)^ku_\alpha^mdu_\alpha=\oint_\gamma \left(g_\Gamma(v_\alpha )\right)^ku_\alpha^mdu_\alpha=\oint_{\gamma'}v_\alpha^ku_\alpha^mdu_\alpha=0\]
where $\gamma'\subset\Sigma'$ is obtained by parallel transporting $\gamma\subset\Sigma$ via the foliation.  The final equality uses the holomorphicity of $v_\alpha^ku_\alpha^mdu_\alpha$.  Parallel transport to a holomorphic differential defined along a different fibre is an analogous mechanism to equation \eqref{vanres} in the proof of Proposition~\ref{phisec1}.
\end{proof}
    
    
    
    \begin{appendices}
    \addtocontents{toc}{\protect\setcounter{tocdepth}{1}}
        \section{Variational Formulae}
\label{variation}
Recall from Section~\ref{foliation} that correlators of a curve embedded in a foliated symplectic surface $\Sigma\subset(X,\Omega_X,\cf)$ are defined recursively via \eqref{rec} given by
\begin{align}  \label{rec1}
\omega_{h,n}(p_1, p_S) = \sum_{du(\alpha) = 0} \mathop{\text{Res}}_{p=\alpha}\ & K(p_1, p) \Bigg[ \omega_{h-1,n+1}(p, \sigma_\alpha(p), p_S) \\
&+\hspace{-2mm} \mathop{\sum_{h_1+h_2=h}}_{I \sqcup J = S} \omega_{h_1,|I|+1}(p, p_I) \, \omega_{h_2,|J|+1}(\sigma_\alpha(p), p_J) \Bigg]\nonumber
\end{align}
where $\Sigma$ enters via the recursion kernel $K = K(p_1,p)$ for $p_1 \in \Sigma$ and $p$ in the vicinity of a ramification point defined by
\begin{equation*}
    K(p_1,p) := -\frac{1}{2}\frac{\int_{p' = \sigma_\alpha(p)}^{p' = p}\omega_{0,2}(p_1,p')}{\omega_{0,1}(p) - \omega_{0,1}(\sigma_\alpha(p))}.
\end{equation*}
which is globally defined in $p_1$ in $p$.  It satisfies 
\begin{equation}\label{recursiveapprox}
    K(p_1,p) \sim_{p \approx \alpha} -\frac{1}{2}\frac{B(p_1,p)}{dv(p)du(p)} + \mathrm{holomorphic}.
\end{equation}
In this appendix we prove a variational formula for the topological recursion correlators with respect to vector fields on $\mathcal{B}$.  We begin first with the Rauch variational formula.

\subsection{Rauch variational formula}
Let $\mathcal{B}$ be a family of curves $\Sigma$ embedded in a foliated symplectic surface $(X, \Omega_X, \mathcal{F})$. Choose FD coordinates $(u,v)$ on $X$ in a neighbourhood of $\alpha\in R\subset\Sigma\subset X$ satisfying $(u,v)(\alpha)=(0,0)$.  

\begin{lemma}  \label{Rauch}
The variation of the Bergman kernel $B(p,q)$ on a curve $\Sigma$ in the family $\mathcal{B}$ is given by
\begin{equation}\label{rauchvariationalformula}
    \nabla^\cf_{\frac{\partial}{\partial z^i}}B(p,q) = -\sum_{\alpha\in R}\Res_{r=\alpha}\frac{\omega_i(r)B(p,r)B(r,q)}{du(r)dv(r)}.
\end{equation}
\end{lemma}
This formula appeared in various places before \cite{fay1992kernel, KKoNew,eynard_orantin, BHuSpe} but we will provide the version of a proof which works in our setting. 
\begin{proof}
A local Rauch variational formula gives the variation of the Bergman kernel with respect to critical values of a locally defined function.  Choose local FD coordinates $(u_\alpha,v_\alpha)$ in a neighbourhood \(U_\alpha \subset X\) of $\alpha$ satisfying the properties of Definition~\ref{specoord}.  Recall the map $\Lambda:U_\Sigma\to\bc^R$ defined in \eqref{coordtocritval}.
For $[\Sigma']\in U_\Sigma$, define a local coordinate $z_\alpha$ (up to $\pm1$) on $\Sigma'$ by
\begin{equation}  \label{zalpha}
  u|_{\Sigma'}=z_\alpha^2+\lambda_\alpha  
\end{equation}
so that $v|_{\Sigma'}=v(z_\alpha)$.  
Then
$$0=\nabla^\cf_i u=2z_\alpha\nabla^\cf_iz_\alpha+\nabla^\cf_i\lambda_\alpha,\quad\Rightarrow\nabla^\cf_i\lambda_\alpha=-2z_\alpha\nabla^\cf_iz_\alpha.$$
Hence the normalised holomorphic differential satisfies
$$\omega_i=-\nabla^\cf_i(vdu)=-(\nabla^\cf_iv)du=-v'(z_\alpha)(\nabla^\cf_iz_\alpha)du=\frac{v'(z_\alpha)}{2z_\alpha}\frac{\partial\lambda_\alpha}{\partial z^i}2z_\alpha dz_\alpha=\frac{\partial\lambda_\alpha}{\partial z^i}dv.
$$
Thus the linearisation $D\Lambda:\bc^g\to\bc^R$ is given by
\begin{equation}  \label{Lambdalin}
    \frac{\partial}{\partial z^i}\lambda_\alpha(z^1,...,z^g)=\left(\frac{\omega_i}{dv}\right)(\alpha).
\end{equation} 
The local Rauch variational formula \cite{RauWei, KKoNew} is
\begin{equation}\label{rauchvariation}
    \nabla_{\frac{\partial}{\partial\lambda_\alpha}}^{\mathcal{F}}B(p,q) = \Res_{r=\alpha}\frac{B(p,r)B(r,q)}{du_\alpha(r)}
\end{equation}
hence
$$\nabla^\cf_{\frac{\partial}{\partial z^i}}B(p,q) = \sum_{\alpha\in R}\frac{\omega_i(r)}{dv_\alpha(r)}\Res_{r=\alpha}\frac{B(p,r)B(r,q)}{du(r)}
$$
and since the zero $du(\alpha)=0$ is simple, \eqref{rauchvariationalformula} follows.
\end{proof}

\subsection{Variation of correlators}

Eynard and Orantin proved a formula for the variation of topological recursion correlators $\omega_{h,n}$ in \cite{eynard_orantin}.  We include the proof here for completeness since the definition of a spectral curve in this paper is slightly different to that in \cite{eynard_orantin}.
\begin{proposition}
For $\Sigma\subset X$ and $\frac{\partial}{\partial z^i}\in T_{[\Sigma]}\mathcal{B}$: 
\begin{equation}\label{omegagnvariationformula}
    \nabla^\cf_{\frac{\partial}{\partial z^i}}\omega_{h,n}(p_1,\cdots,p_n) = -\frac{1}{2\pi i}\oint_{p_{n+1}\in b_i}\omega_{h,n+1}(p_1,\cdots,p_n,p_{n+1})
\end{equation}
\end{proposition}
\begin{proof}
The proof of this formula uses the Rauch variational formula in Lemma~\ref{Rauch} and follows exactly the proof in \cite[Theorem 5.1]{eynard_orantin}.  We will prove it by induction on $2h-2+n$.  The basic idea is simple---apply $\nabla^\cf_{\frac{\partial}{\partial z^i}}=\nabla^\cf_i$ to \eqref{rec1}.  Most terms of the covariant derivative are obtained immediately from the inductive hypothesis and it remains to understand variation of the kernel $K(p_1,p)$.   


Rewrite (\ref{rauchvariationalformula}) as follows:
\begin{align*}
    \nabla^\cf_i B(p,q) &= -\sum_{\alpha \in R}\Res_{r=\alpha}\frac{B(p,r)}{dv(r)du(r)}B(r,q)\omega_i(r)\\
    &= \sum_{\alpha \in R}\Res_{r=\alpha}\frac{B(p,r)}{dv(r)du(r)}B(\sigma_\alpha(r),q)\omega_i(r)\\
    &= -2\sum_{\alpha \in R}\Res_{r=\alpha}K(p,r)B(\sigma_\alpha(r),q)\omega_i(r)\\
    &= -\sum_{\alpha \in R}\Res_{r=\alpha}K(p,r)\left(B(\sigma_\alpha(r),q)\omega_i(r) + B(r,q)\omega_i(\sigma_\alpha(r))\right)
\end{align*}
where the second equality uses the fact that $B(r,q)+B(\sigma_\alpha(r),q)$ vanishes at $r=\alpha$ which cancels the simple pole of the integrand, the third equality uses $\displaystyle\Res_{r=\alpha}\tfrac{B(p,r)}{dv(r)du(r)}f(r)= -2\Res_{r=\alpha}K(p,r)f(r)$ for $f$ holomorphic at $\alpha$ and the final equality uses symmetry.

To simplify the notation, in a neighbourhood of $\alpha\in\Sigma$  define
\begin{equation*}
    E_q(p) := -\frac{1}{2}\int_{q' = \sigma_\alpha(q)}^{q' = q}\omega_{0,2}(p,q'), \qquad \Omega(p) := vdu(p) -vdu (\sigma_\alpha(p))
\end{equation*}
so that $K(p,q) = \frac{E_q(p)}{\Omega(q)}$. By integrating (\ref{rauchvariationalformula}) from $q' = \sigma_\alpha(q)$ to $q' = q$ along a contour that does not intersect the ramification point $r_\alpha$, we have
\begin{align*}
    \nabla^\cf_i E_q(p) &= 2\sum_{\alpha \in R}\Res_{r=\alpha}K(p,r)E_q(r)\omega_i(r)\\
    &= -\sum_{\alpha \in R}\Res_{r=\alpha}K(p,r)(E_q(\sigma_\alpha(r))\omega_i(r) + E_q(r)\omega_i(\sigma_\alpha(r))).
\end{align*}
If $f = f(q)$ is any function then we have \cite[Lemma 5.1]{eynard_orantin}:
\begin{align*}
    \nabla^\cf_i &\left(\sum_{\alpha \in R}\Res_{q=\alpha}K(p,q)f(q)\right) = \sum_{\alpha \in R}\Res_{q=\alpha}\frac{E_q(p)}{\Omega(q)}\nabla^\cf_i f(q)\\
     +& \sum_{\alpha\in R}\Res_{q=\alpha}\Big( 2\sum_{\beta \in R}\Res_{r=\beta}\frac{E_r(p)}{\Omega(r)}\frac{E_q(r)}{\Omega(q)}\omega_i(r)f(q)- \frac{E_q(p)}{(\Omega(q))^2}(\omega_i(q) - \omega_i(\sigma_{\alpha}(q)))f(q)\Big)\\
    &= \sum_{\alpha \in R}\Res_{q=\alpha}\frac{E_q(p)}{\Omega(q)}\nabla^\cf_i f(q) - \sum_{\alpha\in R}\Res_{q=\alpha}\frac{E_q(p)}{(\Omega(q))^2}(\omega_i(q) - \omega_i(\sigma_{\alpha}(q)))f(q)\\
    & + 2\sum_{\alpha \in R}\left(\sum_{\beta \in R}\Res_{r=\beta}\Res_{q=\alpha} - \Res_{q=\alpha}\Res_{r=q} - \Res_{q=\alpha}\Res_{r=\sigma_\alpha(q)}\right) \frac{E_r(p)}{\Omega(r)}\frac{E_q(r)}{\Omega(q)}\omega_i(r)f(q)\\
    &= \sum_{\alpha \in R}\Res_{q = \alpha}K(p,q)\nabla^\cf_i f(q) + 2\sum_{\alpha, \beta \in R}\Res_{r=\alpha}\Res_{q=\beta}K(p,r)K(r,q)\omega_i(r)f(q)\\
    &= \sum_{\alpha \in R}\Res_{q = \alpha}K(p,q)\nabla^\cf_i f(q)\\
    &- \sum_{\alpha, \beta \in R}\Res_{r=\alpha}\Res_{q=\beta}\Big(K(p,r)\big(K(\sigma_\beta(r),q)\omega_i(r)
    + K(r,q)\omega_i(\sigma_\beta(r))\big)f(q)\Big).
\end{align*}
We are now ready to prove the variational formula (\ref{omegagnvariationformula}), which we will do so by induction on $2h-2+n$. The base case $(h,n) = (0,2)$ uses
the Rauch variational formula:
\begin{align*}
    \nabla^\cf _i\omega_{0,2}(p_1,p_2) &= \nabla^\cf _iB(p_1,p_2) = -\sum_{\alpha \in R}\Res_{r=\alpha}\frac{\omega_i(r)B(r,p_1)B(r,p_2)}{du_\alpha(r)dv_\alpha(r)}\\
    &= -2\sum_{\alpha \in R}\Res_{r=\alpha}K(p_1,r)\omega_{0,2}(r,p_2)\left(\frac{1}{2\pi i}\oint_{p_3\in b_i}\omega_{0,2}(r,p_3)\right)\\
    &= -\frac{1}{2\pi i}\oint_{p_3\in b_i}\omega_{0,3}(p_1,p_2,p_3),
\end{align*}
where we have used (\ref{recursiveapprox}) together with the recursive formula for $\omega_{0,3}$ in the second-last and last equalities respectively. 

Proceeding via induction, given $(h,n)$ we shall assume that (\ref{omegagnvariationformula}) holds for all $(h',n')$ such that $2h'-2+n' < 2h-2+n$. Then by applying $\nabla^\cf_i $ to \eqref{rec1} we have
\begin{align*}
    &\nabla^\cf _i\omega_{h,n}(p_1,\cdots,p_n) = -\frac{1}{2\pi i}\sum_{\alpha, \beta \in R}\Res_{r=\alpha}\Res_{p=\beta}K(p_1,r)\\
    &\qquad\times\left(K(\sigma_\alpha(r),p)\oint_{p_{n+1}\in b_i}\omega_{0,2}(r,p_{n+1}) + K(r,p)\oint_{p_{n+1}\in b_i}\omega_{0,2}(\sigma_\alpha(r),p_{n+1})\right)\\
    &\quad\times\left(\omega_{h-1,n+1}(p,\sigma_\alpha(p),p_2,\cdots,p_n) + \hspace{-8mm}\sum_{\substack{h_1 + h_2 = h\\I_1\coprod I_2 = \{2,...,n\}}}\hspace{-8mm}\omega_{h_1,1 + |I_1|}(p,p_{I_1})\omega_{h_2,1+|I_2|}(\sigma_\alpha(p),p_{I_2})\right)\\
    &\qquad - \frac{1}{2\pi i}\sum_{\alpha \in R}\Res_{p = \alpha}K(p_1,p)\Bigg(\oint_{p_{n+1}\in b_i}\omega_{h-1,n+2}(\sigma_\alpha(p), p, p_2,\cdots, p_n, p_{n+1})\\
    &\qquad\qquad + \sum_{\substack{h_1 + h_2 = h\\I_1\coprod I_2 = \{2,...,n\}}}\oint_{p_{n+1}\in b_i}\omega_{h_1,|I_1|+2}(p,p_{I_1},p_{n+1})\omega_{h_2,|I_2|+1}(\sigma_\alpha(p),p_{I_2})\\
    &\qquad\qquad + \sum_{\substack{h_1 + h_2 = h\\I_1\coprod I_2 = \{2,...,n\}}}\omega_{h_1,|I_1|+1}(p,p_{I_1})\oint_{p_{n+1}\in b_i}\omega_{h_2,|I_2|+2}(\sigma_\alpha(p),p_{I_2},p_{n+1})    
    \Bigg)\\
    &= -\frac{1}{2\pi i}\oint_{p_{n+1}\in b_i}\sum_{\alpha \in R}\Res_{r=\alpha}K(p_1,p)\Bigg(\omega_{h-1,n+2}(\sigma_\alpha(p),p,p_2,\cdots,p_n,p_{n+1})\\
    &\quad + \omega_{h,n}(\sigma_\alpha(p),p_2,\cdots,p_n)\omega_{0,2}(p,p_{n+1}) + \omega_{h,n}(p,p_2,\cdots,p_n)\omega_{0,2}(\sigma_\alpha(p),p_{n+1})\\
    &\qquad + \sum^{*}_{\substack{h_1 + h_2 = h\\I_1\coprod I_2 = \{2,...,n,n+1\}}}\omega_{h_1,|I_1|+1}(p,p_{I_2})\omega_{h_2,|I_2|+1}(\sigma_\alpha(p),p_{I_2})\Bigg)\\
    &= -\frac{1}{2\pi i}\oint_{p_{n+1}\in b_i}\omega_{h,n+1}(p_1,\cdots,p_n,p_{n+1}).
\end{align*}
Where \small{$\displaystyle\sum^{*}$} \normalsize indicates that we exclude all terms involving $\omega_{0,2}(.,p_{n+1})$ from the summation.
\end{proof}
    
    \end{appendices}
    
    \printbibliography
     
\end{document}